\documentclass{amsart}
\usepackage[a4paper, margin=2.3cm]{geometry}

\usepackage[utf8]{inputenc}
\usepackage{amsmath}
\usepackage{amssymb}
\usepackage{amsthm}
\usepackage{graphicx}
\usepackage[dvipsnames]{xcolor}
\usepackage{enumitem}
\usepackage{tikz}
\usetikzlibrary{arrows.meta, positioning, calc}
\usepackage{tikz-cd}
\usepackage{subcaption}
\usepackage{csquotes}
\usepackage[T1]{fontenc}
\usepackage{hyperref}
\usepackage{geometry}
\usepackage{comment}
\usepackage[colorinlistoftodos,prependcaption,color=yellow,textsize=tiny,textwidth=2cm]{todonotes}
\usepackage{epigraph}
\usepackage{bbm}

\newcommand{\eps}{\varepsilon}

\newcommand{\dd}{\,\mathrm{d}}

\renewcommand{\Re}{\mathop{\mathrm{Re}}}
\renewcommand{\Im}{\mathop{\mathrm{Im}}}

\newcommand{\calb}{{\mathcal B}}

\newcommand{\cald}{{\mathcal D}}

\newcommand{\calf}{{\mathcal F}}

\newcommand{\call}{{\mathcal L}}

\newcommand{\caln}{{\mathcal N}}
\newcommand{\calo}{{\mathcal O}}

\newcommand{\cals}{{\mathcal S}}

\newcommand{\R}{\mathbb{R}}

\newcommand{\Z}{\mathbb{Z}}
 \newcommand{\C}{\mathbb{C}}
\newcommand{\N}{\mathbb{N}}

\DeclareMathOperator{\id}{\mathsf{id}}

\DeclareMathOperator{\supp}{supp}

\DeclareMathOperator{\sgn}{sgn}
\DeclareMathOperator{\ord}{ord}
\DeclareMathOperator{\conv}{conv}
\DeclareMathOperator{\ad}{ad}

\newcommand{\opA}{\mathsf{A}}
\newcommand{\opB}{\mathsf{B}}

\newcommand{\opL}{\mathsf{L}}

\newcommand{\set}[1]{\ensuremath{\{#1\}}}

\newcommand{\seqkN}[1]{\ensuremath{(#1_k)_{k\in\N}}}

\newcommand{\op}{\mathsf{op}}
\newcommand{\Tr}{\mathsf{Tr}}
\newcommand{\ic}{\mathsf{i}}

\newcommand{\opnorm}[1]{{\lvert\kern-0.25ex\lvert\kern-0.25ex\lvert #1 \rvert\kern-0.25ex\rvert\kern-0.25ex\rvert}}

\newcommand{\CR}[1]{C^{#1}}

\newcommand{\CRi}{\CR \infty}
\newcommand{\CRci}{\CR \infty_c}

\newcommand{\newCCtr}[2][d]{
\newcounter{#2}\setcounter{#2}{0}
\expandafter\xdef\csname kyedtheconst#2\endcsname{#1}
}
\newcommand{\Cc}[2][nolabel]{
\stepcounter{#2}
\expandafter\ensuremath{\csname kyedtheconst#2\endcsname_{\arabic{#2}}}
\ifthenelse{\equal{#1}{nolabel}}
{}
{\expandafter\xdef\csname kyedconst#1\endcsname
{\expandafter\ensuremath{\csname kyedtheconst#2\endcsname_{\arabic{#2}}}}}
}
\newcommand{\Ccn}[2][nolabel]{
\expandafter\ensuremath{\csname kyedtheconst#2\endcsname}
\ifthenelse{\equal{#1}{nolabel}}
{}
{\expandafter\xdef\csname kyedconst#1\endcsname
{\expandafter\ensuremath{\csname kyedtheconst#2\endcsname}}}
}

\newcommand{\Cclast}[1]{
\expandafter\ensuremath{\csname kyedtheconst#1\endcsname_{\arabic{#1}}}
}

\newcommand{\Ccllast}[1]{
\addtocounter{#1}{-1}
\expandafter\ensuremath{\csname kyedtheconst#1\endcsname_{\arabic{#1}}}
\addtocounter{#1}{1}
}
\newcommand{\const}[1]{
\expandafter{\ifcsname kyedconst#1\endcsname
  \csname kyedconst#1\endcsname
\else
  \errmessage{Undefined Kyedconstant #1.}%
\fi}
}

\theoremstyle{plain}
\newtheorem{theorem}{Theorem}[section]
\newtheorem{proposition}[theorem]{Proposition}
\newtheorem{lemma}[theorem]{Lemma}
\newtheorem{corollary}[theorem]{Corollary}
\theoremstyle{definition}
\newtheorem{definition}[theorem]{Definition}
\newtheorem{example}[theorem]{Example}

\newtheorem{remark}[theorem]{Remark}

\title[Time-Periodic Cahn-Hilliard-Gurtin System on the Half Space]{The Time-Periodic Cahn-Hilliard-Gurtin System on the Half Space as a Mixed-Order System with General Boundary Conditions}
\author{Guillaume~Neuttiens}
\address{Guillaume~Neuttiens\newline Friedrich-Schiller-Universit\"at Jena \newline Inselplatz 5, 07743 Jena, Germany}
\email{guillaume.neuttiens@uni-jena.de}
\author{Jonas~Sauer}
\address{Jonas~Sauer\newline Friedrich-Schiller-Universit\"at Jena \newline Inselplatz 5, 07743 Jena, Germany}
\email{jonas.sauer@uni-jena.de}

\begin{document}


\begin{abstract}
A well-posedness and maximal regularity result for the time-periodic Cahn-Hilliard-Gurtin system in the half space is proved.
For this purpose, we introduce a novel class of complementing boundary conditions, extending the classical Lopatinski\u{\i}-Shapiro conditions from elliptic and parabolic theory to time-periodic mixed-order systems with general boundary conditions.
Moreover, we show that the classical Lopatinski\u{\i}-Shapiro conditions are in general insufficient for well-posedness of mixed-order systems.
\end{abstract}

\dedicatory{(On the occasion of the 60th birthday of Robert Denk.)}

\maketitle
\section{Introduction}
This paper develops a comprehensive $L^2$ theory for the time-periodic Cahn-Hilliard-Gurtin system on the half space.
We consider the system
    \begin{align}\label{eqn: CHG_system_1}
        \left\{\begin{array}{rcll}
        \partial_tu_1 -\Delta u_2 &=&f_1 & \text{in } \mathbb{T}\times \mathbb{R}^n_+,      \\
        \Delta u_1-\partial_tu_1+u_2 &=&f_2  & \text{in } \mathbb{T}\times \mathbb{R}^n_+,\\
        \partial_n u_1|_{\set{x_n=0}}&=&g_1  & \text{on } \mathbb{T}\times \mathbb{R}^{n-1},\\
        \partial_n u_2|_{\set{x_n=0}}&=&g_2 &\text{on } \mathbb{T}\times \mathbb{R}^{n-1},
    \end{array}\right.
    \end{align}
where all functions are periodic in time with period $T>0$, and accordingly the time variable lives in the torus $\mathbb{T}:=\R/T\Z$.
Such systems were introduced in \cite{Gur96} and arise naturally in the study of phase separation processes and interface dynamics when microscopic relaxation effects are taken into account.
The model augments the classical Cahn–Hilliard equation by coupling it to an additional evolution equation for the chemical potential, which leads to a symbol that is neither homogeneous nor quasi-homogeneous.
This prevents the use of classical elliptic or parabolic theory and makes the analysis considerably more delicate.
Wilke \cite{Wil12} has studied the Cahn-Hilliard-Gurtin system in the initial value setting, providing a maximal $L^p$ regularity result in the whole space, the half space and on sufficiently smooth bounded domains.

\medskip

The present work provides the first maximal regularity result for system \eqref{eqn: CHG_system_1} in the time periodic setting.
Our approach reveals that the Cahn-Hilliard-Gurtin system is a representative example of a broader class of mixed order systems with general boundary conditions.
These systems are characterised by the presence of several competing scaling behaviours in their symbols, a phenomenon that can be captured in a precise geometric way by means of Newton polygons.
Building on the general symbolic framework developed in \cite{denk2013general,denk1998newton,DeV02,gindikin2012method}, we introduce a refined theory of potential spaces that is adapted to boundary value problems in the time-periodic context and that reflects the mixed order structure of the system.
The first main result asserts that the operator associated with system \eqref{eqn: CHG_system_1} is an isomorphism between natural Newton polygon spaces for the solution and the data.
We refer to the next section for notational conventions and the definition of the involved function spaces, but already mention that $\overline{H}^{(s,r)}_\perp(\mathbb{T}\times \R^n_+)$ is the restriction to the half space of a purely oscillatory Sobolev space of mixed smoothness monitoring independently the $L^2$-norm of up to $s$ derivatives in time and $r$ derivatives in space, and $H^{(s,r)}_\perp(\mathbb{T}\times \R^{n-1})$ is its counterpart on the boundary $\mathbb{T}\times\R^{n-1}$.
\begin{theorem}\label{thm: main}
    Consider the solution space $\mathbb{E}:=\mathbb{E}_1\times \mathbb{E}_2$, the data space $\mathbb{F}:=\mathbb{F}_1\times\mathbb{F}_2$, and the trace space $\mathbb{G}:=\mathbb{G}_1\times\mathbb{G}_2$, given by
    \begin{align*}
    \begin{array}{rclrcl}
        \mathbb{E}_1&:=&\overline{H}^{(1,1)}_\perp(\mathbb{T}\times \mathbb{R}^n_+)\cap \overline{H}^{(0,3)}_\perp(\mathbb{T}\times \mathbb{R}^n_+),
        &\mathbb{E}_2&:=&\overline{H}^{(0,2)}_\perp(\mathbb{T}\times \mathbb{R}^n_+), \\
        \mathbb{F}_1&:=&L^2_\perp(\mathbb{T}\times \mathbb{R}^n_+),
        &\mathbb{F}_2&:=&\overline{H}^{(0,1)}_\perp(\mathbb{T}\times \mathbb{R}^n_+),\\
        \mathbb{G}_1&:=&H_\perp^{(\frac34,0)}(\mathbb{T}\times \mathbb{R}^{n-1})\cap H^{(0, \frac{3}{2})}_\perp(\mathbb{T}\times \mathbb{R}^{n-1}), &\mathbb{G}_2&:=&H^{(0, \frac{1}{2})}_\perp(\mathbb{T}\times \mathbb{R}^{n-1}).
    \end{array}
    \end{align*}
        For all $f=(f_1, f_2)\in \mathbb{F}$ and $g=(g_1,g_2)\in \mathbb{G}$, system \eqref{eqn: CHG_system_1} admits a unique solution $u=(u_1, u_2)\in\mathbb{E}$, and it holds 
\begin{equation*}
    \|u\|_{\mathbb{E}}\lesssim \|f\|_{\mathbb{F}}+\|g\|_{\mathbb{G}}, 
\end{equation*}
where the implicit constant depends only on the time period $T$ and the dimension $n$.
\end{theorem}
We emphasize that Theorem \ref{thm: main} indeed constitutes a maximal regularity result in the time-periodic framework: If we introduce $S:\mathbb{E}\to \mathbb{F}\times\mathbb{G}$ via
\begin{align*}
S(u_1,u_2)=(\partial_tu_1 -\Delta u_2, \Delta u_1-\partial_tu_1+u_2, \partial_n u_1|_{\set{x_n=0}}, \partial_n u_2|_{\set{x_n=0}}),
\end{align*}
then Theorem \ref{thm: main} shows that $S$ is an isomorphism of Banach spaces.

\medskip

A central idea of our approach is a factorization of the determinant into two operators that preserve the support in the lower and upper half space, respectively.
This separates the analysis into a factor that is inverted without prescribing traces, and a factor for which solvability requires boundary traces of the appropriate order. In this way, the factorization isolates the part of the symbol that is responsible for boundary conditions.
The idea of using the Paley-Wiener theorem for such a factorization was first introduced for elliptic scalar equations in \cite{Ark67} and can also be found in \cite[Section 5.3]{triebel1995interpolation}.
In \cite{KyS17, kyed2019time}, it was shown that this approach is also applicable for (time-periodic) parabolic scalar equations.
A core contribution of the present paper is to demonstrate that the factorization method can be adapted to treat general boundary value problems on the half space which are both of mixed-order and vector-valued.
In addition, we develop precise trace results for Newton polygon spaces.
Together, these tools allow us to identify precise algebraic obstructions that prevent certain boundary operators from yielding well-posedness on the half space.

\medskip

The celebrated results by Agmon-Douglis-Nirenberg \cite{agmon1959estimates,agmon1964estimates} show that a cornerstone of the theory of well-posedness for elliptic problems with boundaries are the \emph{Lopatinski\u{\i}-Shapiro conditions}, also known as Agmon complementing conditions.
For parabolic equations, the Lopatinski\u{\i}-Shapiro conditions continue to be sufficient for a comprehensive well-posedness theory in the $L^p$ framework as shown in \cite{DHP07} for the initial value problem and in \cite{kyed2019time} for the time-periodic case.
Let us briefly summarize the idea behind the Lopatinski\u{\i}-Shapiro conditions:
Consider a differential operator $\op[P]$ whose symbol $P: \mathbb{R}\times \mathbb{R}^n\to \mathbb{C}$ is a polynomial, and assume that for every $(\tau, \xi')\in (\mathbb{R}\setminus\set{0})\times \mathbb{R}^{n-1}$, the complex polynomial $z\to P(\tau, \xi', z)$ is of order $2N$ and has no real roots.
Then under appropriate ellipticity or parabolicity conditions one can write $P=P^+P^-$, where $P^+(\tau,\xi',\cdot)$ has all its $N$ roots in the upper half plane and $P^-(\tau,\xi',\cdot)$ has all $N$ roots in the lower half plane.
The Lopatinski\u{\i}-Shapiro condition requires that the boundary operators be independent modulo the factor $P^+$.
In particular, Dirichlet type conditions prescribing $u|_{x_n=0}, \partial_n u|_{x_n=0}, \ldots,\partial_n^{N-1}u|_{x_n=0}$ provide a basic class of boundary operators covered by this framework.
For parabolic operators, these conditions are sufficient and often necessary to guarantee the existence of a unique solution in the maximal regularity space.
For operators of mixed order however, the picture turns out to be more subtle.
Our second main theorem shows the insufficiency of the Lopatinski\u{\i}-Shapiro conditions for the operator
\begin{align*}
    \mathsf{D}:=\partial_t-\Delta(\partial_t-\Delta).
\end{align*}
This operator is closely related to the Cahn-Hilliard-Gurtin system \eqref{eqn: CHG_system_1}.
Indeed, $\mathsf{D}$ can be understood as the determinant of the Cahn-Hilliard-Gurtin system, and it will be paramount in our analysis that it \emph{does} provide a well-posedness theory on both the whole space and the half space if endowed with appropriate boundary conditions, see Example \ref{def_ellipticity_general} and Corollary \ref{cor: corollary_trace_1_3}.
However, for pure Dirichlet conditions, Theorem \ref{thm: main_2} shows that solvability is governed not by the boundary data alone, but by a compatibility condition involving the inverse of the factor $D^-$ corresponding to the roots of $D(\tau,\xi',\cdot)$ in the lower half plane.
\begin{theorem}\label{thm: main_2}
    Let $f\in L^2_\perp(\mathbb{T}\times\R^n_+)$.
    Consider the spaces
    \begin{align*}
        \mathbb{E}&:=\overline{H}^{(1,2)}_\perp(\mathbb{T}\times \R^n_+)\cap \overline{H}^{(0,4)}_\perp(\mathbb{T}\times \R^n_+), \\
        \mathbb{G}&:=H^{(\frac34,0)}_\perp(\mathbb{T}\times\R^{n-1})\cap H^{(0,\frac32)}_\perp(\mathbb{T}\times\R^{n-1}).
    \end{align*}
    The Dirichlet problem 
    \begin{equation}\label{eqn: CHG_Dirichlet}
        \left\{\begin{array}{rcll}
             \partial_t u-\Delta(\partial_t-\Delta) u&=&f & \text{in } \mathbb{T}\times \mathbb{R}^n_+,  \\
         u|_{x_n=0}&=&0 & \text{on } \mathbb{T}\times \mathbb{R}^{n-1}, \\
        \partial_n u|_{x_n=0}&=&0 & \text{on } \mathbb{T}\times \mathbb{R}^{n-1}   
        \end{array}\right.
    \end{equation}
    admits a solution $u\in \mathbb{E}$ if and only if $(\op[D^-]^{-1}_+f)|_{x_n=0}\in \mathbb{G}$.
    In particular, there exists $f\in L^2_\perp (\mathbb{T}\times\R^n_+)$ such that \eqref{eqn: CHG_Dirichlet} does not admit a solution in $\mathbb{E}$. 
\end{theorem}
Here, $\op[m]$ stands for the Fourier multiplier operator with multiplier $m$, while the notation $\opA_+$ is related to the half space and introduced in Definition \ref{def: supp_pres_op} below.
Theorem \ref{thm: main_2} demonstrates that an algebraic \emph{complementing boundary condition} adapted to Newton polygon geometry must be strictly stronger than the Lopatinski\u{\i}-Shapiro conditions in order to guarantee well-posedness for symbols with non-triangular Newton polygons.
We will provide precisely such a complementing boundary condition in Definition \ref{def:compl_cond} by demanding an appropriately extended boundary matrix containing the information of $D^+$ and the boundary operators to be a mixed order system on the boundary $\mathbb{T}\times\R^{n-1}$.
The stronger nature of our condition compared to Lopatinski\u{\i}-Shapiro is captured by a clear algebraic formulation, see Proposition \ref{prop: LScond}.
Namely, we find that the Lopatinski\u{\i}-Shapiro conditions are equivalent to the invertibility of the extended boundary matrix, which is strictly weaker than being a mixed order system.

\medskip

To keep the presentation streamlined, we restrict ourselves to the cases required for \eqref{eqn: CHG_system_1}.
In particular, we show the trace theorem (see Theorem \ref{thm: trace}) only for Newton polygons with four vertices \linebreak $\set{(0,0),(r_1,0),(r_2,s_2),(0,s_2)}$.
Moreover, we only study the $\mathcal{N}$-ellipticity of the roots of the Cahn-Hilliard-Gurtin determinant, not of roots corresponding to general $\mathcal{N}$-elliptic symbols which are polynomial in the variable $\xi_n$.
Despite these restrictions, the underlying arguments apply more broadly, and we expect that the general case can be developed within the framework introduced here.
We postpone this to future work.

\medskip

The structure of the paper is as follows.
Section \ref{sec:pre} recalls the theory of Newton polygons and order functions, and introduces the Newton polygon potential spaces adapted to the time–periodic setting.
Section \ref{sec:whole} establishes maximal regularity for mixed-order systems on the whole space, with the Cahn-Hilliard-Gurtin system serving as the guiding example.
Section \ref{sec:half} develops the boundary analysis on the half space, including support preserving operators and trace theorems for Newton polygon spaces.
Finally, Section \ref{sec:proof} applies the abstract theory to the system \eqref{eqn: CHG_system_1} and presents the proofs of Theorems \ref{thm: main} and \ref{thm: main_2}.

\section{Preliminaries and Notations}\label{sec:pre}
Throughout this work the time period $T>0$, and the dimension $n\in\N$ are fixed.
The notion $A\lesssim B$ denotes an estimate of the form $A\le CB$, where the constant $C$ depends only on $T$ and the dimension $n$.
If the constant $C$ depends on additional quantities $\Lambda$, we will write $A\lesssim_\Lambda B$.
We will write $A\simeq B$ if both $A\lesssim B$ and $B\lesssim A$.
The imaginary unit is denoted by $\ic$.
Given two locally convex spaces $E, F$ we will write $\mathcal{L}(E, F)$ for the space of all bounded linear operators $E\to F$ and $\call_{\mathrm{iso}}(E,F)$ the spaces of bounded, invertible operators $E\to F$ with bounded inverse.
Both spaces are endowed with the usual bounded-open topology.
In particular, if  $E$ and $F$ are normed spaces, the topology on $\call(E,F)$ is given by the operator norm. If $E=F$, we write $\call(E)$ instead of $\call(E,E)$.

\medskip

We use the convention that for a set $Y$, a topological space $X$ and a map $\varphi:X\to Y$, we denote by $\varphi X$ the image space $\varphi(X)$ equipped with the final topology.
If additionally $Y$ is a vector space, $X$ a (locally convex) topological vector space, and $\varphi$ linear with a closed kernel, then $\varphi X$ is a (locally convex) topological vector space as well, see e.g.\@ \cite[Theorem 2.4.3]{Eng89} and \cite[Chapter 4 and Proposition 7.9]{Tre67}.
Similarly, if $X$ is a Fr\'echet or normed space, then so is $\varphi X$, see \cite[Chapters 10 and 11]{Tre67}.
In particular, if $X$ is a Banach space, then so is $\varphi X$, and its norm is given by the quotient norm
\begin{align*}
\|f\|_{\varphi X}=\inf\{\|F\|_{X}\mid F\in X, \ \varphi(F)=f\}.
\end{align*}
For a continuous map $\varphi:X\to Y$ between two topological vector spaces, we define its transpose ${}^t \varphi:Y'\to X'$ via $\langle {}^t \varphi(y'),x\rangle:=\langle y',\varphi(x)\rangle$ for all $y'\in Y'$ and $x\in X$.
Then ${}^t \varphi$ is continuous as well if $X'$ and $Y'$ carry both the weak or both the strong dual topology \cite[Proposition 19.5]{Tre67}.

\medskip

For $\xi\in\R^n$, we write $\langle \xi\rangle:=(1+|\xi|^2)^{\frac12}$.
Two Banach spaces $A_0$ and $A_1$ are called compatible if they embed into a common Hausdorff space.
For two such compatible couples $\set{A_0,A_1}$ and $\set{B_0,B_1}$ we write $\call(\set{A_0,A_1},\set{B_0,B_1})$ for the set of linear operators $T:A_0+A_1\to B_0+B_1$ such that $T|_{A_0}\in \call(A_0,B_0)$ and $T|_{A_1}\in \call(A_1,B_1)$.
For $\theta\in (0,1)$, the complex interpolation is denoted by $[A_0,A_1]_\theta$ and the real interpolation with fine index $2$ by $(A_0,A_1)_{\theta,2}$.
We emphasize that if $A_0$ and $A_1$ are Hilbert spaces, then
\begin{align}\label{eqn: equiv_real_complex}
 [A_0,A_1]_\theta=(A_0,A_1)_{\theta,2}
\end{align}
with equivalent norms, see e.g. \cite[Corollary C.4.2]{HNV16}.
Moreover, we recall the following result on the interplay of interpolation and retraction.
\begin{lemma}\label{lem:interp_retraction}
    Let $\set{A_0,A_1}$ and $\set{B_0,B_1}$ be couples of compatible Banach spaces.
    If $R_k\in \call(A_k,B_k)$ is a retraction with corresponding coretraction $E_k\in \call(B_k,A_k)$ such that $R_0x=R_1x$ for all $x\in A_0\cap A_1$, then for all $\theta\in (0,1)$ it holds
    \begin{align*}
        R[A_0,A_1]_\theta=[B_0,B_1]_\theta
    \end{align*}
    with equivalent norms, where $R:A_0+A_1\to B_0+B_1$ is given by $R(a_0+a_1)=R_0a_0+R_1a_1$.
\end{lemma}
\begin{proof}
    Follows from Theorem 1.2.4 in \cite{triebel1995interpolation}, see e.g. Lemma 1.51 in \cite{denk2013general}.
\end{proof}
\subsection{Newton Polygons and Order Functions}
As mentioned in the introduction, the fact that the symbol
\begin{equation}\label{eqn: determinant}
    D(\tau, \xi)=\ic\tau+\vert \xi\vert(\ic\tau+\vert \xi\vert), \qquad (\tau,\xi)\in \mathbb{R}\times \mathbb{R}^n,
\end{equation}
associated to the Cahn-Hilliard-Gurtin determinant is neither homogeneous nor quasi-homogeneous presents major difficulties since the usual methods related to elliptic or parabolic equations often rely on the study of homogeneities.
To tackle this problem, the notion of a Newton polygon has been introduced, see \cite{denk2013general, denk1998newton}.
This theory allows to treat symbols presenting themselves as a sum of products of terms with various homogeneities.
Here we recall only the results we need for our problem and refer to \cite{denk2013general} for an in-depth exposition of the theory.

\begin{definition}
    Let $E=\{(x_1, y_1),...,(x_M, y_M)\}\subseteq [0, \infty)^2$ be a finite set. The Newton polygon of $E$, denoted $\mathcal{N}(E)$ is the convex envelope of the elements of $E$ and of their projections on the coordinate axis
    \begin{equation*}
        \mathcal{N}(E):=\conv\left (E\cup \bigcup_{m=1}^M \{(x_m, 0), (0, y_m)\} \cup\{(0,0\}\right). 
    \end{equation*}
\end{definition}
Throughout this section we consider a polynomial 
\begin{equation}\label{eqn: polynom}
    P(\tau, \xi)= \sum_{i=1}^{I}\sum_{\substack{\alpha\in \mathbb{N}^n\\ \vert \alpha\vert\leq  N}} p_{i, \alpha} \tau^i\xi^\alpha, \qquad (\tau,\xi)\in \mathbb{R}\times \mathbb{R}^n.
\end{equation}
The Newton polygon of $P$ is a geometric representation of the degree of the monomials occuring in the multivariate polynomial $P(\tau, \xi)$.
It is constructed by considering the set of the couples $(i, \alpha)$ associated to the degree of each non-zero monomial.
\begin{definition}
    Let $P$ be as in (\ref{eqn: polynom}) and define $E(P):= \{ (i, \vert \alpha\vert) \mid p_{i, \alpha} \neq 0\}$. Then the Newton polygon associated to $P$ is $\mathcal{N}(P):=\mathcal{N}(E(P))$.
\end{definition}
Given some finite set $E\subseteq [0, \infty)$ and $\mathcal{N}=\mathcal{N}(E)$ its Newton polygon, we denote 
\begin{equation}\label{eqn: vertices}
    \mathcal{N}_V:=\{(r_0, s_0),..., (r_{J+1}, s_{J+1}) \}
\end{equation}
the set of its vertices, ordered counter-clockwise starting from $(r_0, s_0):=(0,0)$. Remark that by definition, $r_1$ is the degree of the polynomial $P(\tau, \xi)$  in the variable $\xi$ while $s_{J+1}$ is the degree of the polynomial in the variable $\tau$.
We say that $\mathcal{N}$ is \emph{regular in time} if $r_2\neq r_1$.

\begin{example}\label{ex: CHD_polygon}
Let us consider the symbol $D$ of the Cahn-Hilliard-Gurtin determinant (\ref{eqn: determinant}). The set $E(D)$ writes 
\begin{equation*}
    E(D)=\{ (4,0), (2, 1), (0,1) \}
\end{equation*}
and so 
\begin{equation}\label{eqn : CHG_det_polygon}
    \mathcal{N}_V(D)=\{(0,0), (4,0), (2, 1), (0,1) \}.
\end{equation}
   
\end{example}
 \medskip
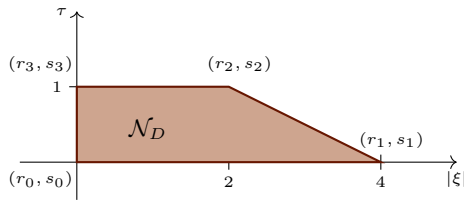
\begin{figure}[h]
\centering
\begin{tikzpicture}
 \draw[->] (-.75,0) -- (5,0) node[below , font=\tiny] {$\vert \xi\vert$};
  \draw[->] (0,-.5) -- (0,2) node[left, font=\tiny] {$\tau$}; 
  
   \draw (2,0.1) -- (2,-0.1);
  \draw (4,0.1) -- (4,-0.1);
    \node[below, yshift= -2pt, font=\tiny] at (2,0) {$2$};
    \node[below, yshift=-2pt, font=\tiny] at (4, 0) {$4$};

    \draw (-0.1, 1) -- (0.1, 1);
    \node[left, xshift=-2pt, font= \tiny] at (0,1) {$1$};
  \coordinate (A) at (0,0);
  \coordinate (B) at (4,0);
  \coordinate (C) at (2,1);
  \coordinate (D) at (0,1);

  \filldraw[color=Sepia, fill=Sepia!30,thick] (A)--(B)--(C)--(D)--cycle;

     \node[below left, xshift=2pt,  font=\tiny] at (0,0) {$(r_0, s_0)$};
    \node[above, xshift=4pt, yshift=2pt, font=\tiny] at (4,0) {$(r_1, s_1)$};
     \node[above, xshift=4pt, yshift=2pt, font=\tiny] at (2,1) {$(r_2, s_2)$};
      \node[above left, xshift=2pt, yshift=2pt, font=\tiny] at (0,1) {$(r_3, s_3)$};

      \node[font=\small] at ($(A)!0.25!(B)!0.25!(C)!0.25!(D)$) {$\mathcal{N}_D$};
\end{tikzpicture}
\caption{The Newton polygon of $D$, where $r_0=r_3=s_0=s_1=0$, $r_1=4$, $r_2=2$, and $s_2=s_3=1$.}
\end{figure}
As mentioned above, the concept of the Newton polygon aims at providing a geometric description of the degree of multivariate polynomials $P(\tau, \xi)$. As such, one would expect the Newton polygon of a product to be related to the 'sum' of the Newton polygons of each factor. However, the notion of sum of polygons is not completely intuitive. Instead, it is often much easier to introduce the notion of order functions.
\begin{definition}\label{def_order_function}
    An \emph{order function} is a function $\mu: (0, \infty)\to \mathbb{R}$ that is continuous and piecewise linear. For such a function,  there exists $J\in \mathbb{N}$, $0=:\gamma_0 \leq \gamma_1 <\gamma_2 <...< \gamma_J\leq \gamma_{J+1}:=\infty$ and $b_1,..., b_{J+1}, m_1,..., m_{J+1}\in \R$ such that 
    \begin{equation}
        \mu(\gamma)= b_j+\gamma m_j \qquad \forall \gamma\in [\gamma_{j-1}, \gamma_j), \quad j\in\set{1,..., J+1}.
    \end{equation}
\end{definition}
If $\mathcal{N}$ is a Newton polygon with vertices $\mathcal{N}_V$ ordered as in (\ref{eqn: vertices}), the associated order function is the order function $\mu_{\mathcal{N}}$ defined by 
\begin{equation}\label{eqn: order_function}
    \mu_{\mathcal{N}}(\gamma):= \max\set{r_j+\gamma s_j\mid j\in\set{1,\ldots,J+1}}.
\end{equation}
Due to the convexity of the Newton polygon $\mathcal{N}$, one can show (see \cite{denk2013general} Remark 2.20 ) that the function $\mu$ is always convex and positive. An order function presenting itself as (\ref{eqn: order_function}) with $r_j\geq 0, s_j\geq 0$ for all $j=0,..., J+1$ is called \emph{strictly positive order function}. By extension, if an order function $\nu$ is such that  $-\nu$ is strictly positive, we say that $\nu$ is a \emph{strictly negative order function}.

For $j\in\set{1,...,J}$, we define 
\begin{equation}\label{eqn: def_gamma_j}
    \gamma_j:=\frac{r_j-r_{j+1}}{s_{j+1}-s_j},
\end{equation}
where we understand $\gamma_J=\infty$ if $s_J=s_{J+1}$. Geometrically, $\gamma_j$ is the slope of the normal vector to the vertex $[(r_j, s_j)(r_{j+1}, s_{j+1})]$. We also set $\gamma_0:=0$ and $\gamma_{J+1}=\infty$. By convexity of the Newton polygon, one obtains 
$0=\gamma_0\leq \gamma_1< \gamma_2< ...<\gamma_J\leq \gamma_{J+1}=\infty$. Remark that $\mathcal{N}$ is regular in time if and only if $\gamma_1> 0$. 

We get the following description of the order function $\mu$. See \cite{denk2013general} Section 2.2 for the proof.
\begin{lemma}
    Let $\mathcal{N}$ be a Newton polygon and let $\gamma_0,..., \gamma_{J+1}$ be defined as above.
    Then for $j\in\set{1,..., J+1}$ and $\gamma\in [\gamma_{j-1}, \gamma_{j})$, 
    \begin{equation*}
        \mu_{\mathcal{N}}(\gamma)= \gamma s_j+r_j.
    \end{equation*}
\end{lemma}

Conversely, let $\mu$ be a strictly positive order function. By definition, there exists $K\in \mathbb{N}$, $\delta_0, \delta_1,...,\delta_{K+1}$ with $0=: \delta_0\leq \delta_1<\delta_2<...<\delta_{K}\leq \delta_{K+1}:=\infty $ and $b_1,..., b_{K+1}\geq 0, m_1,...,m_{K+1}\geq 0$ such that $\mu(\gamma)=b_k+\gamma m_k$ for all $\gamma\in [\delta_{k-1}, \delta_k)$. 

To the strictly positive order function above, we associate the set 
\begin{equation*}
    E(\mu):=\{(0,0)\}\cup\{(b_k, m_k), (b_k, 0), (m_k, 0) \mid k=1,...K+1\}.
\end{equation*}
The Newton polygon associated to $\mu$ is $\mathcal{N}(\mu):=\conv(E(\mu))$. One easily sees that there is a one-to-one correspondence between Newton polygons and order functions in the sense that $\mathcal{N}(\mu_\mathcal{N})=\mathcal{N}$ for any Newton polygon $\mathcal{N}$ and $\mu_{\mathcal{N}(\mu)}=\mu$ for any strictly positive order function $\mu$. It is however important to notice that order functions are more general than Newton polygons, since our Definition \ref{def_order_function} allows us to consider also order functions that are not strictly positive, i.e. that are not associated to any Newton polygon.

For any strictly positive order function $\mu$ and its Newton polygon $\mathcal{N}(\mu)$ with a set of vertices 
$\mathcal{N}(\mu)_V=\{(r_0, s_0),..., (r_{J+1}, s_{J+1})\}$, we define the \emph{order} of $\mu$ by $\ord\mu:=r_1$.

We now introduce the concept of weight functions. For a symbol $p(\xi)$ depending only on $\xi$, a classic way to define ellipticity is to ask for a lower estimate of the form $\vert p\vert \geq C \vert \xi\vert^d$ with $d$ the degree of the polynomial. The weight functions associated to the Newton polygons generalise this idea by providing two-sided estimates for our multivariate polynomials $P(\tau, \xi)$.  

\begin{definition}
    Let $\mu$ be an order function with notations as in Definition \ref{def_order_function}.
    \begin{enumerate}
        \item If $\mu$ is a strictly positive order function with associated Newton polygon $\mathcal{N}(\mu)$ and its set of vertices $\mathcal{N}(\mu)_V:=\{(r_0, s_0), (r_1, s_1),...(r_{J+1}, s_{J+1})\}$ ordered counter-clockwise starting from $(r_0, s_0)=(0,0)$.
        The \emph{weight function} associated to $\mu$ is $W_\mu: \mathbb{R}\times \mathbb{R}^n\to \mathbb{R}$ defined by 
        \begin{equation*}
            W_\mu(\tau, \xi):=1+\sum_{j=1}^{J+1}\vert \tau\vert^{s_j}\vert \xi\vert^{r_j}, \qquad \tau\in \mathbb{R}, \xi\in \mathbb{R}^n.
        \end{equation*}
        \item   If $\mu$ is a strictly negative order function, then by definition $-\mu$ is strictly positive and we set 
        \begin{equation*}
            W_\mu(\tau, \xi):=\frac{1}{W_{-\mu}(\tau, \xi)}. 
        \end{equation*}
    \end{enumerate}
\end{definition}
The vertices of a Newton polygon can be thought of as the points of 'maximal weight', as explained in the following Lemma.
\begin{lemma}\label{lemma: point_inside}
    Let $\mathcal{N}_1$ be a Newton polygon and $\mu_1:=\mu(\mathcal{N}_1)$ the associated order function. Then $(r, s)\in \mathcal{N}_1$ if and only if there exists a constant $C>0$ such that for all $(\tau, \xi)\in \mathbb{R}\times \mathbb{R}^n$ it holds
    \begin{equation*}
        \vert \tau\vert^s\vert\xi\vert^r \leq C W_{\mu_1}(\tau, \xi). 
    \end{equation*}
    In particular, if $\mathcal{N}_2$ is another Newton polygon such that $\mathcal{N}_1\subseteq \mathcal{N}_2$ and $\mu_2:=\mu(\mathcal{N}_2)$ is its order function, then there exists $C>0$ such that $W_{\mu_1}\leq C W_{\mu_2}$.
\end{lemma}
\begin{proof}
    See \cite[Remark 2.17]{denk2013general}.
\end{proof}
Since the purpose of order functions is to compute easily sums and differences of Newton polygons, we are also interested in how the weight functions behave with respect to sums and differences of order functions. We have the following property.
\begin{lemma}\label{lemma: additivity}
    If $\mu, \nu$ are both strictly positive order functions or both strictly negative order functions and $\alpha,\beta>0$, then $W_{\alpha\mu+\beta\nu}\simeq W_{\mu}^\alpha W_\nu^\beta$.
\end{lemma}
\begin{proof}
    See \cite[Lemma 2.33]{denk2013general} for $\alpha=\beta=1$.
    The general case follows from this, since $W_{\alpha \mu}\simeq W_\mu^{\alpha}$ for $\alpha>0$ by Lemma \ref{lemma: point_inside}.
\end{proof}
In view of the above lemma, one can also define weight functions for the difference of order functions in the following way.
Let $\mu, \nu$ be two strictly positive order functions.
We set 
\begin{equation*}
   W_{\mu-\nu}:\simeq \frac{W_\mu}{W_\nu}.
\end{equation*}
Here $W_{\mu-\nu}$ is well-defined up to the equivalence  $\simeq$. Indeed, let $\mu_1, \mu_2, \nu_1, \nu_2$  be strictly positive order functions such that $\mu_1-\nu_1=\mu_2-\nu_2$. Then $\mu_1+\nu_2=\mu_2+\nu_1$ are both strictly positive order functions and by the above lemma, $W_{\mu_1}W_{\nu_2}\simeq W_{\mu_2}W_{\nu_1}$.

\medskip

The weight function associated to a strictly positive order function can be decomposed as a product of 'elementary' weight functions using the slopes of the edges of the associated Newton polygon $\mathcal{N}(\mu)$. This will be particularly useful in the last sections for practical computations.
More precisely, for any $y\geq 0$, we define the elementary order functions $o_y(\gamma):=\max\{1, y\gamma \}, \gamma>0$.
In particular, for $(\tau, \xi)\in \mathbb{R}\times \mathbb{R}^n$ it holds 
    \begin{equation*}
        W_{o_{y}}(\tau, \xi) =1 +\vert \xi\vert+\vert \tau\vert^y 
    \end{equation*}
\begin{proposition}\label{prop: decomposition_elementary_order}
    Let $\mu$ be a strictly positive order function, $\mathcal{N}=\mathcal{N}(\mu)$ the associated Newton polygon and  $\mathcal{N}_V=\{(r_0, s_0),...,(r_{J+1}, s_{J+1})\}$ the set of its vertices. Define $\gamma_1,..., \gamma_{J+1}$ as in (\ref{eqn: def_gamma_j}) and assume that $\mathcal{N}$ is regular in time, that is, $\gamma_1>0$.  Then
    \begin{equation*}
        W_\mu \simeq \prod _{j=1}^{J}W_{o_{1/\gamma_j}}^{r_j-r_{j+1}},
    \end{equation*}
    where we understand $\frac{1}{\gamma_J}=0$ if $\gamma_J=\infty$. 
\end{proposition}
\begin{proof}
    See Chapter I, Subsection 1.3 in \cite{gindikin2012method}.
\end{proof}
\begin{example}\label{ex: CHD_order_fct}
    We continue the discussion for the Cahn-Hilliard-Gurtin determinant from Example \ref{ex: CHD_polygon}.
    Its order function $\mu_D:=\mu(\caln(E(D))$ is given by $\mu_D(\gamma)=\max\set{4,2+\gamma}$, which means that $\mu_D=2o_\frac12+2$.
    The corresponding weight function is $W_D(\tau,\xi):=1+|\xi|^4+|\tau||\xi|^2+|\tau|$.
\end{example}
Knowing if the symbol $P$ can be estimated from above and from below by its weight function will turn out to be crucial to establish maximal regularity estimates.
We do not restrict ourselves to polynomials in the following definition, which will be of importance in Section \ref{sec:half}.
\begin{definition}\label{def: upper_order_function_general}
     Let $P: (\mathbb{R}\setminus \{0\})\times \mathbb{R}^n\to \mathbb{C}$ be smooth and let $\mu$ be the difference of strictly positive order functions.
     We say that $\mu$ is an \emph{upper order function} for $P$ if there exists $C>0$ such that
    \begin{equation*}
        \vert P(\tau, \xi)\vert \leq CW_{\mu}(\tau, \xi)
    \end{equation*}
    for all $(\tau, \xi)\in (\mathbb{R}\setminus\set{0})\times \mathbb{R}^n$.
\end{definition}
We observe that if $\mu_1$ is an upper order function for $P_1$ and $\mu_2$ is an upper order function for $P_2$ in the above sense, then $\mu_1+\mu_2$ is an upper order function for $P_1P_2$.
\begin{definition}\label{def_ellipticity_general}
    Let $P: \mathbb{R}\setminus \{0\}\times \mathbb{R}^n\to \mathbb{C}$ be smooth and let $\mu$ be the difference of strictly positive order functions.
    We say that $P$ is $\mathcal{N}$-elliptic for the order function $\mu$ if 
    \begin{enumerate}
    \item $\mu$ is an upper order function for $P$.
        \item For all $\lambda>0$, there exists $C_\lambda>0$ such that 
        \begin{equation*}
            \vert P(\tau, \xi)\vert \geq C_\lambda W_\mu(\tau, \xi)
        \end{equation*}
        for all $(\tau, \xi)\in \mathbb{R}\times \mathbb{R}^n$ with $\vert \tau \vert \geq \lambda$. 
    \end{enumerate}
\end{definition}
As above, the notion of $\caln$-ellipticity is compatible with taking products of symbols: If $P_j$ is $\caln$-elliptic for the order function $\mu_j$ for $j\in\set{1,2}$, then $P_1P_2$ is $\caln$-elliptic for the order function $\mu_1+\mu_2$.
We immediately show that the CHG determinant is $\mathcal{N}$-elliptic.
\begin{proposition}\label{prop: det_elliptic}
    The symbol $D(\tau, \xi)= i\tau+\vert \xi\vert^2(i\tau+\vert \xi\vert^2)$ is $\mathcal{N}$-elliptic for the associated order function $\mu_D$ defined in Example \ref{ex: CHD_order_fct}. 
\end{proposition}
\begin{proof}
    The upper bound follows from \cite[Lemma 2.30]{denk2013general}.
    Recall that $W_{\mu_D}(\tau, \xi)=1+\vert \xi\vert^4 +\vert \tau\vert\vert \xi\vert^2+\vert\tau\vert$ for all $(\tau, \xi)\in \mathbb{R}\times \mathbb{R}^n$.
    Therefore, we have 
    \begin{align*}
        \vert D(\tau, \xi)\vert &=\sqrt{\vert \tau\vert^2+2 \vert \xi\vert^2\vert \tau\vert^2+\vert \xi\vert^4\vert \tau\vert^2+\vert \xi\vert^8}
        \geq \sqrt{\vert \tau\vert^2+\vert \xi\vert^4\vert \tau\vert^2+\vert \xi\vert^8}
        \geq \frac{1}{\sqrt{3}}(\vert \tau\vert+\vert \xi\vert^2\vert \tau\vert +\vert \xi\vert^4).
    \end{align*}
    Let $\lambda>0$. Then we obtain for $|\tau|\ge \lambda$ the desired estimate with $C_{\lambda}:=\frac{\min\set{1,\lambda}}{2\sqrt3}\le \frac{1}{\sqrt3}$ due to $|\tau|\ge \frac12\min\set{1,\lambda}(1+|\tau|)$.
\end{proof}
\subsection{Mixed-Order Systems}
The notion of mixed-order systems has been introduced in \cite{agmon1964estimates} in order to study the regularity of solutions of systems of elliptic equations. The definition was later adapted to the Newton polygon setting in \cite{denk2013general}.
In the following, we consider a matrix-valued symbol  $L:(\mathbb{R}\setminus \{0\})\times \mathbb{R}^n\to \mathbb{C}^{m\times m}$ for some $m\in \mathbb{N}$. We denote by $\mathcal{D}: (\mathbb{R}\setminus\{0\})\times \mathbb{R}^n\to \mathbb{C}$, $\mathcal{D}(\tau, \xi):=\det L(\tau, \xi)$ its determinant.
\begin{definition}\label{def: mixed_order_systems}
    The symbol $L$ is associated to a \emph{mixed-order system} if there exist strictly positive order functions $-s_1,...,-s_m, t_1,...,t_m$ such that the following is true.
    \begin{enumerate}
        \item The order function $s_i+t_j$ is an upper order function of the symbol $L_{ij}: \mathbb{R}\times \mathbb{R}^n\to \mathbb{C}$ for all $i,j\in\set{1,..., m}$.
        \item The determinant $\mathcal{D}$ is $\mathcal{N}$-elliptic for the order function $\delta:=\sum_{k=1}^m s_k+t_k$.
    \end{enumerate}
\end{definition}
While solving systems, we will make a regular use of the adjugate symbol of $L$, denoted $\ad L$. Here 
\begin{equation*}
    (\ad L)_{ij}:= \det L^{(ji)}, \qquad i,j\in\set{1,.., m}, 
\end{equation*}
where $L^{(ji)}$ denotes the $(m-1)\times (m-1)$ matrix obtained by removing the $j$-th line and the $i$-th column of $L$.
\begin{proposition}\label{prop: ad}
    If $L:\mathbb{R}\setminus \{0\}\times \mathbb{R}^n\to \mathbb{C}$ is a mixed-order system for the order functions $s_1,..., s_m$, $t_1,...,t_m$ then $\ad L$ is a mixed-order system for the order functions $S_1,...,S_m$, and $T_1,...,T_m$, with 
    \begin{equation*}
        S_i=\sum_{\substack{k=1\\k\neq i}}^m s_k, \qquad T_j:=\sum_{\substack{k=1\\k\neq j}}^m t_k,\qquad i, j\in\set{1,...,m}.
    \end{equation*}
\end{proposition}
\begin{proof}
    See part (I) of the proof of Theorem 2.69 in \cite{denk2013general}.
\end{proof}
\begin{example}\label{ex: CHG_mixed_order_system}
    One can give the CHG system 
    \begin{equation}\label{eqn: CHG_system}
        \left\{\begin{array}{rcl}
        \partial_tu_1 -\Delta u_2 &=&f_1,       \\
        \Delta u_1-\partial_tu_1+u_2 &=&f_2,
        \end{array}\right.
    \end{equation}
    a structure of mixed-order system. Indeed, the symbol matrix associated to the problem is 
    \begin{equation}\label{eqn: CHG_symbol}
        L(\tau, \xi)=
        \begin{pmatrix}
        \ic\tau & \vert \xi\vert^2\\
        -\vert \xi\vert^2-\ic\tau & 1
        \end{pmatrix}, \qquad (\tau, \xi)\in \mathbb{R}\times \mathbb{R}^n.
    \end{equation}
    Defining the order functions 
    \begin{align*}
        &t_1(\gamma):=\max\{3, 1+\gamma\}, \quad
        t_2(\gamma):=2, \quad
        s_1(\gamma):=0, \quad
        s_2(\gamma):=-1,
    \end{align*}
    we obtain that $s_i+t_j$ is an upper order function for the symbol $L_{ij}$ for all $i, j\in\set{1,2}$.
    Furthermore, one has 
    \begin{equation*}
        \delta(\gamma):= (s_1+s_2+t_1+t_2)(\gamma)=\max \{4, 2+\gamma\}=\mu_D(\gamma)
    \end{equation*}
    with $D(\tau, \xi)=\det L(\tau, \xi)$.
    Since we established in Proposition  \ref{prop: det_elliptic} that the determinant $D$ is $\mathcal{N}$-elliptic for the order function $\mu_D$, we indeed have a mixed-order system. 
\end{example}

\subsection{Time-Periodic Setting and Purely Oscillatory Function Spaces}
We start by recalling the time-periodic setting as introduced e.g.~in \cite{EiK17,KyS17,kyed2019time}.
Any time-periodic function $f:\mathbb{R}\to \mathbb{C}$ of period $T>0$ can be considered as a function defined on the torus $\mathbb{T}:=\mathbb{R}/ T\mathbb{Z}$. Thus our aim is to solve (\ref{eqn: CHG_system}) for some functions $u_1, u_2, f_1, f_2$ defined on $G:=\mathbb{T}\times \mathbb{R}^n$.\\
The differential structure on $G$ is the one inherited from $\mathbb{R}\times \mathbb{R}^n$ through the quotient projection $\pi: \mathbb{R}\times \mathbb{R}^n\to G$, that is, 
\begin{equation*}
    C^\infty(G)=\{f: G\to \mathbb{C} \mid f\circ \pi \in C^\infty(\mathbb{R}\times \mathbb{R}^n)\}
\end{equation*}
and for all $f\in C^\infty (G)$, $k\in \mathbb{N}, \alpha\in \mathbb{N}_0^n$, 
\begin{equation*}
    (\partial_t^k\partial_x^\alpha f)(\pi(t, x)):=\partial_t^k\partial^\alpha_x (f\circ \pi)(t, x).
\end{equation*}
For $K\subseteq \R^n$ we define $\CRi_K(G):=\set{f\in \CRi(G)\mid \supp f\subseteq \mathbb{T}\times K}$, which is turned into a Fr\'echet space in the usual way via the seminorms
\begin{align*}
 \rho_{m,K}(f):=\sup\set{\sum_{k+|\alpha|\le m} |\partial_t^k\partial_x^\alpha f(t,x)| \mid (t,x)\in \mathbb{T}\times K}.
\end{align*}
Fix an exhaustive sequence $K_1\subseteq K_2\ldots\subseteq \R^n$ of compact sets.
Then we can define the test function space $\CRci(G):=\bigcup_{j\in\N} \CRi_{K_j}(G)$ as endowed with the canonical LF topology, and introduce the set of distributions $\cald'(G)$ as the strong dual of $\CRci(G)$.
Since derivatives are continuous operators on $\CRci(G)$, the notion of distributional derivatives can be introduced by transposition in the usual way.
Integration on $G$ is done with respect to the product of the Haar measure on $\mathbb{T}$ and of the Lebesgue measure on $\mathbb{R}^n$.
In particular, it holds  $f\in L^2(G)$ if and only if $f\circ \pi_{\vert [0, T]\times \mathbb{R}^n}\in L^2([0, T]\times \mathbb{R}^n)$. 

Once the weak derivatives and integration are defined, one can consider the usual Sobolev spaces. In particular, we are interested in distinguishing the regularity  in the variables $t$ and $x$. We thus introduce for $s,r\in\N$ the Sobolev space of dominating mixed smoothness
\begin{equation}\label{eqn: def_sobolev}
    H^{(s, r)}(G):=\{f\in L^2(G)\mid \partial_t^s\partial_x^\alpha f, \partial_x^\alpha f, \partial_t^sf\in L^2(G) \text{ for all } \alpha\in \mathbb{N}_0^n \text{ with } \vert \alpha\vert=r\}.
\end{equation}

\subsubsection{Fourier Transform on $G$}
The space $G$ defined above is a locally compact abelian group and thus admits a Pontryagin dual 
\begin{equation*}
    \widehat{G}:=\frac{2\pi}{T}\mathbb{Z}\times \mathbb{R}^n, 
\end{equation*}
where one identifies $\widehat{G}$ by associating to any $(k, \xi)\in \frac{2\pi}{T}\mathbb{Z}\times \mathbb{R}^n $ the character $\chi: G\to \mathbb{C}$, $\chi(t, x):=\mathrm{e}^{\ic x\cdot \xi+\ic tk}$, see e.g. \cite[Section 1.2]{Rud62}.
The measure on $\widehat{G}$ will be the product of the counting measure on $\frac{2\pi}{T}\mathbb{Z}$ and the Lebesgue measure on $\mathbb{R}^n$.

\medskip

The differentiable structure on $\widehat{G}$ is defined as 
\begin{equation*}
    C^\infty(\widehat{G}):=\{h\in C(\widehat{G})\mid h(k, \cdot)\in C^\infty(\mathbb{R}^n) \text{ for all } k\in \mathbb{Z}\}
\end{equation*}
 where for all $f\in C^\infty(\widehat{G}), (k, \xi)\in \widehat{G}$ and $\alpha\in \mathbb{N}^n$, we set 
 \begin{equation*}
     \partial_x^\alpha f(k, \xi):=(\partial_x^\alpha f(k, \cdot))(\xi).
 \end{equation*}

In a usual fashion, we define the Schwartz-Bruhat space on $G$ by 
\begin{equation*}
    \mathcal{S}(G):=\{f\in C^\infty(G): \rho_{\alpha, \beta, \gamma}(f)<\infty \text{ for all } (\alpha, \beta, \gamma)\in \mathbb{N}_0\times \mathbb{N}_0^n\times \mathbb{N}_0^n\}, 
\end{equation*}
where the topology is induced by the seminorms
\begin{equation*}
    \rho_{\alpha, \beta, \gamma}(f):=\sup\set{\vert x^\gamma \partial_t^\alpha \partial_x^\beta f(t, x)\vert \mid (t, x)\in G}.
\end{equation*}
Similarly, we introduce
\begin{equation*}
    \mathcal{S}(\widehat{G}):=\{h\in C^\infty(\widehat{G}): \hat{\rho}_{\alpha, \beta, \gamma}(h)<\infty \text{ for all } (\alpha, \beta, \gamma)\in \mathbb{N}_0\times \mathbb{N}_0^n\times \mathbb{N}_0^n\},
\end{equation*}
with the seminorms 
\begin{equation*}
    \hat{\rho}_{\alpha, \beta, \gamma}(h):=\sup\set{\vert k^\alpha \xi^\gamma \partial_\xi^\beta h(k, \xi)\vert\mid (k, \xi)\in \widehat{G}}.
\end{equation*}
The tempered distribution spaces  $\mathcal{S}'(G)$ and $\mathcal{S}'(\widehat{G})$) are then introduced as the strong dual of $\mathcal{S}(G)$ and $\mathcal{S}(\widehat{G})$, respectively.

\medskip

Being a locally compact abelian group, $G$ admits a Fourier transform  $\mathcal{F}_G: L^1(G)\to C(\widehat{G})$ defined by 
\begin{equation*}
    \mathcal{F}_Gf (k, \xi)= \int_{\mathbb{T}}\int_{\mathbb{R}^n}f(t, x)\exp(-\ic x\cdot \xi - \ic tk)dxdt, \qquad (k, \xi)\in \widehat{G}.
\end{equation*}
and similarly, $\widehat G$ admits a Fourier transform $\mathcal{F}_{\widehat{G}}: L^1(\widehat{G})\to C(G)$
\begin{equation*}
    \mathcal{F}_{\widehat G}h(t, x):=\sum_{k\in \mathbb{Z}}\int_{\mathbb{R}^n}h(k, \xi)\exp(-\ic x\cdot \xi -\ic tk)d\xi, \qquad (t, x)\in G.
\end{equation*}
It holds $\mathcal{F}_G\in \call_{\mathrm{iso}}(\cals(G),\cals(\widehat{G}))$ and $\mathcal{F}_{\widehat{G}}\in \call_{\mathrm{iso}}(\cals(\widehat{G}),\cals(\widehat{\widehat{G}}))$, see \cite[Theorem 3.2]{Waw68}.
By transposition, the Fourier transform defines an isomorphism $\calf_G:={}^t\mathcal{F}_{\widehat{G}}\in \call_{\mathrm{iso}}(\mathcal{S}'(G), \mathcal{S}'(\widehat{G}))$.
We will often write $\mathcal{F}:=\mathcal{F}_G$.

\medskip
Armed with the Fourier transform on $G$, we can define Fourier multipliers in a usual fashion. 
We denote by $\mathcal{O}(\mathbb{R}\times \mathbb{R}^n)$ the set of symbols $m\in C^\infty (\mathbb{R}\times \mathbb{R}^n)$ such that $m$ and all its derivatives are polynomially bounded. 
We will also set 
\begin{equation*}
    \mathcal{O}(\widehat{G}):=\{ m_{\vert \widehat{G}} \mid m\in \mathcal{O}(\mathbb{R}\times \mathbb{R}^n)\}. 
\end{equation*}
Then for $m\in \mathcal{O}(\widehat{G})$, the multiplication $M_m:\cals(\widehat{G})\to\cals(\widehat{G})$, $M_m(\varphi):=m\varphi$ is linear and bounded, and we may write $mT:={}^t M_m(T)$ for $T\in \cals'(G)$.
Then for $m\in \mathcal{O}(\widehat{G})$, we define $\op[m]\in\call(\mathcal{S}'(G),\mathcal{S}'(G))$ via 
\begin{equation}\label{eqn: def_multiplier}
    \op[m]f :=\mathcal{F}^{-1}m\mathcal{F}f.
\end{equation}
The definition implies in particular that $\op[m_1]\op[m_2]=\op[m_1m_2]=\op[m_2]\op[m_1]$ for all $m_1,m_2\in \calo(\widehat{G})$.
A remark on the situation in $L^2(G)$ is in order:
Since the Fourier transform $\mathcal{F}$ restricted to $L^2(G)$ to an isometric isomorphism $\mathcal{F}\in \call_{\mathrm{iso}}(L^2(G),L^2(\widehat{G}))$ by \cite[Theorem 1.6.1]{Rud62}, the definition in \eqref{eqn: def_multiplier} stays well-defined on $L^2(G)$ if $m\in L^\infty(\widehat{G})$, in which case one has $\op[m]\in \call(L^2(G),L^2(\widehat{G}))$.
\subsubsection{Purely Oscillatory Spaces}\label{sec:oscill}
As explained, for example, in \cite{KyS17}, a way to deal with time-periodic problems is to decompose them into one time-independent problem and one time-dependent problem, called 'purely oscillatory', where one can ignore the singularities of our multipliers at the origin $k=0$.
We introduce the time-averaging projection  $\mathcal{P}\in \call(\mathcal{D}(G))$ and $\mathcal{P}_\perp \in\call(\mathcal{D}(G))$ by
    \begin{equation*}
        \mathcal{P}f(t, x)=\int_{\mathbb{T}}f(s,x)ds, \qquad \mathcal{P}_\perp=\id_G-\mathcal{P}.
    \end{equation*}
Note that $\mathcal{P}$ projects onto time-independent functions, while $\mathcal{P}_\perp$ removes the temporal mean.
It is clear that $\mathcal{P}$ and $\mathcal{P}_\perp$ are complementary continuous projections, and we continue to write $\mathcal{P}$ and $\mathcal{P}_\perp$ for their transposes  ${}^t\mathcal{P},{}^t\mathcal{P}_\perp\in \call(\mathcal{D}'(G))$.
For a locally convex space $E(G)\subseteq \mathcal{D}'(G)$, we will write $E_\perp(G):= \mathcal{P}_\perp E(G)$. 
Introducing $\delta_{\mathbb{Z}}:\frac{2\pi}{T}\Z\to \R$ as 
\begin{equation*}
    \delta_{\mathbb{Z}}(k):=\left\{\begin{array}{ll}
    1     & \text{if } k=0,  \\
    0     & \text{if } k\in \frac{2\pi}{T}\mathbb{Z}\setminus \{0\},
    \end{array}\right.
\end{equation*}
we remark $\mathcal{P}=\op[\delta_{\mathbb{Z}}]$ and $\mathcal{P}_\perp= \op[1-\delta_{\mathbb{Z}}]$ as operators on $\mathcal{S}'(G)$.

\medskip

The main point of those projections is to decompose differential equations. Given an equation of the form $\op[m]u=f$ for some $u, f\in \mathcal{S}'(G)$ and $m\in \mathcal{O}(\hat{G})$, one has 
\begin{equation*}
    \op[m]f=u 
    \quad \Leftrightarrow \quad 
    \left\{\begin{array}{rcl}
    \mathcal{P}\op[m] u&=&\mathcal{P}f,      \\
    \mathcal{P}_\perp \op[m]u&=&\mathcal{P}_\perp f,  
    \end{array}\right. 
    \quad \Leftrightarrow \quad 
    \left\{\begin{array}{rcl}
    \op[\delta_{\mathbb{Z}}m] u&=&\mathcal{P}f    ,  \\
     \op[(1-\delta_{\mathbb{Z}})m]u&=&\mathcal{P}_\perp f.  
    \end{array}\right. 
\end{equation*}
By definition of $\mathcal{P}$, the line associated to $\mathcal{P}$ does not depend anymore on the time and often falls into the scope of classical elliptic theory.
We will thus focus mainly on the part associated to $\mathcal{P}_\perp$, often called \emph{purely oscillatory} problem. In Section 3, we will see that reducing time-periodic problem to the purely oscillatory setting allows us to consider operators $\op[m]$ for which the symbol $m$ may exhibit singularities in $k=0$. 

\medskip

Now consider a cut-off function $\eta\in C^\infty(\mathbb{R})$ with $\mathbbm{1}_{[-\frac{\pi}{T},\frac{\pi}{T}]}\le 1-\eta\le \mathbbm{1}_{[-\frac{2\pi}{T},\frac{2\pi}{T}]}$.
In particular
\begin{equation}\label{eqn: def_cutoff}
        \eta(\tau)=\left\{\begin{array}{ll}
        0     & \text{if } \vert \tau\vert \leq \frac{\pi}{T}  \\
        1     & \text{if } \vert \tau \vert \geq \frac{2\pi}{T}.
        \end{array}\right.
    \end{equation}
Then $\eta_{\vert \frac{2\pi}{T}\mathbb{Z}}=1-\delta_{\mathbb{Z}}$. 
We set
\begin{equation*}
    \mathcal{O}^\perp(\widehat{G}):=\{m_{\vert \widehat{G}}: m\in\CRi((\R\setminus\set{0})\times\R^n) \text{ with } \eta m \in \mathcal{O}(\mathbb{R}\times \mathbb{R}^n) \}.
\end{equation*}
For the sake of shortness, if $m\in \CRi((\R\setminus\set{0})\times\R^n)$ is such that $\eta m\in \mathcal{O}(\mathbb{R}\times \mathbb{R}^n)$, we will write $\widehat{m}:=m_{\vert \widehat{G}}$ every time it is necessary to distinguish the symbol from its restriction to $\widehat{G}$. 
For $\widehat{m}\in \mathcal{O}^\perp(\widehat{G})$, we set 
\begin{equation}\label{eqn: def_multiplier_perp}
    \op[m]: = \op[(1-\delta_\mathbb{Z})\widehat{m} ]:\mathcal{S}'_\perp(G)\to \mathcal{S}'_\perp(G). 
\end{equation}
Observe that this definition coincides with (\ref{eqn: def_multiplier}) if $\widehat{m}\in \mathcal{O}(\widehat{G})$.

\subsubsection{Time-Periodic Function Spaces Associated to the Newton Polygon}
In order to deal efficiently with Fourier multipliers, we will introduce some potential spaces associated to order functions.
Those will be defined using the weight functions of Section 2.2 as potential.
For this purpose, it makes sense to introduce a smoothened version of the weight functions.
 \begin{lemma}\label{lemma: smooth_weight_functions}
     Let $\mu$ be the difference of strictly positive order functions.
     There exists a function $w_\mu\in \calo(\mathbb{R}\times \mathbb{R}^n)$ which is $\mathcal{N}$-elliptic for the order function $\mu$ in the sense of Definition \ref{def_ellipticity_general}.
 \end{lemma}
 \begin{proof}
    Assume $\mu$ is strictly positive, i.e.,  $W_\mu$ writes as $W_\mu(\tau, \xi)=1 +\sum_{j=1}^{J+1} \vert \tau\vert^{s_j}\vert \xi\vert^{r_j} $. Then
    \begin{equation*}
        w_\mu(\tau, \xi):= \sum_{j=1}^{J+1} \langle \tau\rangle^{s_j} \langle\xi\rangle^{r_j}=\sum_{j=1}^{J+1} \left( 1+\vert \tau\vert^2\right)^{\frac{s_j}{2}}\left(1+\vert \xi\vert^2\right)^{\frac{r_j}{2}}
    \end{equation*}
    is as intended.
    If $\nu$ is another strictly positive order function, one sets $w_{\mu-\nu}:=\frac{w_{\mu}}{w_{\nu}}$.
 \end{proof}
We can now define our Newton polygon potential spaces on $G$ as follows.
\begin{definition}\label{def: polygon_space}
    Let $\mu$ be the difference of strictly positive order functions.
    The \emph{Newton polygon potential space} associated to $\mu$ is 
    \begin{align*}
        H^{\mu}(G)&:=\op[w_{-\mu}]L^2(G).
    \end{align*}
    For $s,r\in [0,\infty)$, we write $H^{(s,r)}(G):=H^{\mu_{(s,r)}}(G)$ with $\mu_{(s,r)}=\mu(\caln(\set{(r,s)}))$.
\end{definition}
Observe that by definition, we have $\|f\|_{\mu}=\|\op[w_\mu]f\|_{L^2(G)}$.
Moreover, since for the weight function of $\mu_{(s,r)}$ we have  $W_{\mu_{(s, r)}}(\tau, \xi)=1+\vert \xi\vert^r +\vert \tau\vert^s\vert \xi\vert^r+\vert \tau\vert^s$, we may use a standard cut-off method (see for example \cite{bergh2012interpolation} Theorem 6.2.3) to deduce that the above definition of $H^{(s,r)}(G)$ is consistent with \eqref{eqn: def_sobolev} for $s,r\in \N$.
Furthermore, we can relate general Newton polygon potential spaces to the intersection of Sobolev spaces $H^{(s,r)}(G)$ of dominating mixed smoothness at the vertices of the Newton polygon.
In particular, the geometric representation of the Newton polygon $\mathcal{N}(\mu)$ encodes all the information needed to describe $H^{\mu}_\perp(G)$ in terms of Sobolev spaces. 
\begin{proposition}\label{prop_potential_spaces}
    Let $E\subseteq [0,\infty)^2$ be a finite set, $\mathcal{N}:=\mathcal{N}(E)$ the associated Newton polygon with vertices $\mathcal{N}_V:=\{(r_0, s_0),..., (r_{J+1}, s_{J+1})\}$ and $\mu$ the associated order function. Then 
\begin{equation*}
    \bigcap_{(r, s)\in E}H^{(s,r)}(G)=\bigcap_{i=1}^{J+1}H^{(s_i, r_i)}(G)=H^{\mu}(G)
\end{equation*}
with equivalent norms.
\end{proposition}

\begin{proof}
    Set
    \begin{equation*}
        X:=\bigcap_{(s, r)\in E}H^{(s,r)}(G),\qquad Y:=\bigcap_{i=1}^{J+1}H^{(s_i, r_i)}(G),\qquad Z:=H^{\mu}(G).
    \end{equation*}
    The inclusion $X\subseteq Y$ is trivial.
    To show $Y\subseteq Z$, let $f\in Y$.
    Then 
    \begin{equation*}
        \|f\|_Z=\|\op[w_\mu]f\|_{L^2(G)} \lesssim \sum_{j=1}^{J+1} \|\op[w_{\mu_{(s_j, r_j)}}]f \|_{L^2(G)}= \sum_{j=1}^{J+1}\|f\|_{H^{(s, r)}(G)}
    \end{equation*}
    where we used that 
    \begin{equation*}
        m:= \frac{w_\mu}{\sum_{j=1}^{J+1}w_{\mu_{(s_j, r_j)}}}\in L^\infty(\widehat{G})
    \end{equation*}
    is a multiplier in $L^2_\perp(G)$ by Plancherel's theorem. Thus $Y\subseteq Z$.
    
    \medskip
    
    We now show $Z\subseteq X$.
    Let $f\in Z$.
    Since $E\subseteq \mathcal{N}(E)$, we may use Lemma \ref{lemma: point_inside} and once again Plancherel's theorem to obtain that 
    \begin{equation*}
        \widetilde m:=\frac{w_{\mu_{(s, r)}}}{w_\mu}
    \end{equation*}
    is a multiplier in $L^2_\perp(G)$ for all $(r, s)\in E$.
    Thus 
    \begin{equation*}
        \|f\|_{H^{(s, r)}(G)}=\|\op[w_{\mu_{(s, r)}}]f\|_{L^2(G)}\lesssim \|\op[w_\mu]f\|_{L^2(G)},
    \end{equation*}
    and the conclusion follows.
\end{proof}

Let us also record that the operators $\op[w_\mu]$ associated to the weight functions act as lift operators in the potential spaces.
\begin{proposition}\label{prop: isomorphism_weight_function}
    Let $\mu$ and $\nu$ be differences of strictly positive order functions.
    Then it holds that $\op[w_\nu]\in \call_{\mathrm{iso}}(H^{\mu + \nu}(G),H^{\nu}(G))$ with $\op[w_\mu]^{-1}:=\op[w_{-\mu}]$. 
\end{proposition}
\begin{proof}
    Let $f\in H^{\mu+\nu}(G)$.
    Then
    \begin{equation*}
        \|\op[w_\mu]f\|_{\nu} = \|\op[w_\nu w_\mu]f\|_{L^2(G)}=\|\op[\frac{w_\nu w_\mu}{w_{\mu+\nu}}]\op[w_{\mu+\nu}] f\|_{L^2(G)}.
    \end{equation*}
    Since $m: =\frac{w_\nu w_\mu}{w_{\mu+\nu}}\in L^\infty(\widehat{G})$, we have
    \begin{equation*}
         \|\op[\frac{w_\nu w_\mu}{w_{\mu+\nu}}]\op[w_{\mu+\nu}] f\|_{L^2(G)}\lesssim \|\op[w_{\mu+\nu}] f\|_{L^2(G)} =C\|f\|_{\mu+\nu}.
    \end{equation*}
    The boundedness of the inverse follows by analogy.
\end{proof}
As a corollary, one obtains that all spaces $H^{\mu}(G)$ are isomorphic to one another.
Furthermore, we remark that since $H^{\mu}(G)\subseteq \cald'(G)$, we have defined its purely oscillatory counterpart $H^{\mu}_\perp(G)$ in Section \ref{sec:oscill}.
Since $\mathcal{P}_\perp=\op[1-\delta_\mathbb{Z}]$ commutes with all operators of the form $\op[w_\mu]$, one can replace all instances of a Newton polygon potential space $H^\mu_\perp(G)$ by its purely oscillatory counterpart $H^\mu_\perp(G)$ in Propositions \ref{prop_potential_spaces} and \ref{prop: isomorphism_weight_function}.
\begin{lemma}\label{lem: interp_L2}
Let $\mu_0$ and $\mu_1$ be differences of strictly positive order functions.
For $\theta\in (0,1)$ define $\mu:=(1-\theta)\mu_0+\theta\mu_1$.
Then it holds
\begin{align*}
[H^{\mu_0}(G),H^{\mu_1}(G)]_\theta=H^{\mu}(G) \quad \text{and} \quad [H^{\mu_0}_\perp(G),H^{\mu_1}_\perp(G)]_\theta=H^{\mu}_\perp(G).
\end{align*}
with equivalent norms.
\end{lemma}
\begin{proof}
By the Plancherel's theorem, the Fourier transform is an isometric isomorphism between $H^{\mu}(G)$ and  $L^{2}(\widehat{G}, w_\mu)$, where $L^2(\widehat{G}, w_\mu)$ is the weighted $L^2$ space defined by the norm
\begin{align*}
\|f\|_{L^2(\widehat{G}; w_\mu)}^2:=\int_{\widehat{G}} |f(k,\xi)|^2 w_\mu(k,\xi)^2\dd(k,\xi).
\end{align*}
Observe that $[L^2(\widehat{G}; w_{\mu_0}),L^2(\widehat{G}; w_{\mu_1})]_\theta=L^2(\widehat{G}; w_{\mu_0}^{1-\theta}w_{\mu_1}^\theta)$ with equivalent norms, see e.g. \cite[Theorem 5.5.3]{bergh2012interpolation}.
Since $w_{\mu_0}^{1-\theta}w_{\mu_1}^\theta\simeq w_{\mu}$ by Lemma \ref{lemma: additivity}, the first identity follows in view of Lemma \ref{lem:interp_retraction}.

\medskip

The projection $\mathcal{P}_\perp : H^{\nu}(G)\to H^{\nu}_\perp(G)$ is a retraction for all $\nu\in\set{\mu_0,\mu_1,\mu}$, the right-inverse being given by the inclusion $i : H^{\nu}_\perp(G)\hookrightarrow H^{\nu}(G)$, $i(f):=f$.
Thus, we also obtain the second identity by Lemma \ref{lem:interp_retraction}.
\end{proof}
\section{The Whole Space Problem}\label{sec:whole}
\subsection{Maximal Regularity of the CHG Determinant }
We are now able to state well-posedness results for time-periodic mixed-order scalar equations on the whole space, see Theorem \ref{thm: whole space} below. 
\begin{proposition}\label{prop_boundedness}
    Let $\mu$ and $\nu$ be differences of strictly positive order functions and $\widehat{P}\in \calo^\perp(\widehat{G})$.
    Assume that $\mu$ is an upper order function for $P$.
    Then $\op[P]\in \call(H^{\mu+\nu}_\perp(G),H^{\nu}_\perp(G))$.
\end{proposition}
\begin{proof}
     Since $w_{\mu+\nu}\cong w_\mu w_\nu$, we have
    \begin{equation*}
        m:= (1-\delta_\Z)\frac{Pw_{\nu}}{ w_{\mu+\nu}}\in L^\infty(\widehat{G}),
    \end{equation*}
    so that $m$ is a Fourier multiplier in $L^2_\perp(G)$.
    Therefore, we have for all $f\in H^{\nu}_\perp(G)$
    \begin{align*}
        \|\op[P]f\|_{\nu}&=\|\op[(1-\delta_{\mathbb{Z}})\frac{P w_{\nu}}{w_{\mu+\nu}}]\op[w_{\mu+\nu}]f\|_{L^2(G)}\leq C \|\op[w_{\mu+\nu}]f\|_{L^2(G)}= C \|f\|_{\mu+\nu}.
        \qedhere
    \end{align*}
\end{proof}
\begin{theorem}\label{thm: whole space}
    Let $\mu$ and $\nu$ be differences of strictly positive order functions and $\widehat{P}\in \calo^\perp(\widehat{G})$.
    Assume furthermore that $P$ is $\mathcal{N}$-elliptic for the order function $\mu$.
    Then $\op[P]\in \call_{\mathrm{iso}}(H^{\mu+\nu}_\perp(G),H^{\nu}_\perp(G))$ with $\op[P]^{-1}=\op[\frac{1}{P}]$.
    In particular, for all $f\in H^{\nu}_\perp(G)$, the equation $\op[P]u=f$ has a unique solution $u\in H^{\nu+\mu}_\perp(G)$, and with $c:=\|\op[\frac1{P}]\|_{\call(H^{\nu}_\perp(G), H^{\mu+\nu}_\perp(G))}$ it holds 
    \begin{equation*}
        \|u\|_{\mu+\nu}\le c\|f\|_{\nu}.
    \end{equation*}
\end{theorem}
\begin{proof}
    By Proposition \ref{prop_boundedness} we have $\op[P]\in \call(H^{\mu+\nu}_\perp(G),H^{\nu}_\perp(G))$. 
    Since $P$ is $\mathcal{N}$-elliptic, we have $P(\tau, \xi)\neq 0$ for $\tau\neq 0$. Thus $\widehat{P}\in \mathcal{O}^\perp(\widehat{G})$ implies $\frac{1}{\widehat{P}}\in \mathcal{O}^\perp(\widehat{G})$.
    Furthermore, $P$ being $\mathcal{N}$-elliptic for the order function $\mu$ implies that $-\mu$  is an upper order function for $\frac{1}{P}$ and the conclusion follows again from Proposition \ref{prop_boundedness}.
\end{proof}
\begin{remark}\label{remark: Poincare}
    A trivial consequence of the above theorem is that the operator $\partial_t=\op[\ic\tau]$ is bounded and invertible in Newton polygon potential spaces. In particular, one gets a Poincaré-type inequality in the following sense. Let $\nu$ be a difference of strictly positive order functions and consider the order function  $\id$.
    Then for all $k\in \mathbb{N}, $ $\partial_t^k f\in H^{\nu}_\perp(G)$ if and only if $f\in H^{\nu+k\id}_\perp(G)$, and it holds
    \begin{equation*}
        \|f\|_{\nu}\le \| f\|_{\nu+k \id}\cong_{\nu,k} \|\partial_t^k f\|_{\nu}.
    \end{equation*}
\end{remark}
\begin{corollary}\label{cor: CHG_eqn_whole_space}
    Let $D$ be the CHG determinant in \eqref{eqn: determinant} and $\mu_D$ the associated order function.
    Let furthermore $\nu$ be the difference of strictly positive order functions.
    Then for each $f\in H^\nu_\perp(G)$, there is a unique $u\in H^{\mu+\nu}(G)$ with $\op[D]u=f$, and it holds
        \begin{align*}
            \|u\|_{\mu_D+\nu}\lesssim_{\nu} \|f\|_{\nu}.
        \end{align*}
\end{corollary}
\begin{proof}
    Follows directly from Proposition \ref{prop: det_elliptic} and Theorem \ref{thm: whole space}.
\end{proof}
\subsection{Mixed-Order Systems and Maximal Regularity for the CHG System}
In this section, we prove the counterpart of Theorem \ref{thm: whole space} for mixed-order systems. 
Here we consider a smooth  matrix-valued  symbol $L: (\mathbb{R}\setminus \{0\})\times \mathbb{R}^n\to \mathbb{C}^{m\times m}$ as well as $\mathcal{D}:=\det L: (\mathbb{R}\setminus\{0\}) \times \mathbb{R}^n\to \mathbb{C}$, and denote as usual $\widehat{L}, \widehat{D}$ the restrictions to $\widehat{G}$.
We assume that the symbols $\widehat{L}_{ij}$ belong to the class $\mathcal{O}^\perp(\widehat{G})$.
\begin{theorem}\label{thm: whole_space_systems}
    Let $\widehat{L}$ be as above and assume that $L$ is a mixed-order system in the sense of Definition \ref{def: mixed_order_systems} with strictly positive order functions $t_j$, and $-s_i$, $i,j\in\set{1,..., m}$.
    Assume additionally that $\nu$ is a strictly positive order functions.
    We define 
    \begin{equation*}
        \mathbb{E}:=\prod_{j=1}^m H^{\nu+t_j}_\perp(G) \qquad \text{and} \qquad \mathbb{F}:=\prod_{i=1}^m H^{\nu-s_i}_\perp(G).
    \end{equation*}
    Then $\op[L]\in\call_{\mathrm{iso}}(\mathbb{E},\mathbb{F})$ with $\op[L]^{-1}=\op[\ad L]\op[D^{-1}]$, where $D^{-1}$ has to be understood as the diagonal symbol matrix with all diagonal entries equal to $\frac{1}{D}$.
    In particular, for all $f=(f_1, ..., f_m)\in \mathbb{F}$, the system $\op[L]u=f$ admits a unique solution $u=(u_1,...,u_m)\in \mathbb{E}$, and it holds
    \begin{equation*}
        \sum_{j=1}^m \| u_j\|_{\nu+ t_j}\lesssim_{\nu,s_i,t_i,L,m} \sum_{i=1}^m \| f_i\|_{\nu-s_i}.
    \end{equation*}
\end{theorem}
\begin{proof}
    By definition of a mixed-order system, $s_i+t_j$ is an upper order function for $L_{ij}$.
    Thus by Proposition \ref{prop_boundedness}, $\op[L_{ij}]\in \call(H^{\nu+t_j}_\perp(G),H^{\nu-s_i}_\perp(G))$ for all $i, j\in\set{1,...,m}$. 

\medskip

    In a similar fashion,  the determinant $\mathcal{D}$ is $\mathcal{N}$-elliptic for the order function $\delta:=\sum_{l=1}^m (s_l+t_l)$.
    Thus, by Theorem \ref{thm: whole space}, $\op[\mathcal{D}]^{-1}\in \call(H^{\nu-s_i}_\perp(G),H^{\nu-s_i+\delta}_\perp(G))$ for all $i\in\set{1,...,m}$.
    Furthermore, by Proposition \ref{prop: ad} and \ref{prop_boundedness}, we have 
    \begin{equation*}
        \op[(\ad L)_{ij}]\in \call(H^{\nu-s_i+\delta}_\perp(G),H^{\nu-s_i +\delta -S_i-T_j}_\perp(G))
    \end{equation*}
    for all $i,j\in\set{1,\ldots,m}$, where $S_i$ and $T_j$ are defined in Proposition \ref{prop: ad}.
    Since $\nu-s_i +\delta -S_i-T_j=\nu+t_j$, the boundedness of $\op[\ad L]\op[D^{-1}]$ follows.
    Since $\op[L]$ and $\op[\ad L]\op[D^{-1}]$ are mutually inverse, this achieves the proof. 
\end{proof}
Since we already showed that the Cahn-Hilliard-Gurtin system is a mixed-order system, we obtain time-periodic maximal regularity as a direct corollary of the above theorem. 
\begin{corollary}
    For all $f_1\in L^2_\perp(G), f_2\in H^{(0, 1)}_\perp(G)$, the Cahn-Hilliard-Gurtin system \eqref{eqn: CHG_system} admits a unique solution $(u_1, u_2)$ with $u_1\in H^{(0, 3)}_\perp(G)\cap H^{(1, 1)}_\perp(G)$, $u_2\in H^{(0, 2)}_\perp(G)$. 
\end{corollary}
\begin{proof}
    Set $s_1, s_2, t_1, t_2$ as in Example \ref{ex: CHG_mixed_order_system}. By Proposition \ref{prop_potential_spaces}, we get $H^{-s_1}_\perp(G)=L^2_\perp(G)$, $H^{-s_2}_\perp(G)=H^{(0, 1)}_\perp(G)$ as well as $H^{t_1}_\perp(G)= H^{(0, 3)}_\perp(G)\cap H^{(1, 1)}_\perp(G)$ and $H^{t_2}_\perp (G)=H^{(0, 2)}_\perp(G)$.
    The conclusion follows from Theorem \ref{thm: whole_space_systems}.
\end{proof}
\section{The Half Space Problem with General Boundary Conditions}\label{sec:half}
In this section, we develop a theory for mixed order systems with general boundary conditions which will enable us to solve the Cahn-Hilliard-Gurtin system in the half space.
We denote by $\mathbb{R}^n_+=\{x\in \mathbb{R}^n : x_n>0\}$ the upper half space, and write by analogy $\mathbb{R}^n_-=\{x\in \mathbb{R}^n : x_n<0\}$ for the lower half space.
Then we define
\begin{align*}
    G_\pm:=\mathbb{T}\times \mathbb{R}^n_\pm \quad \text{and} \quad G^{n-1}:=\mathbb{T}\times \mathbb{R}^{n-1}.
\end{align*}
Function spaces on $G_\pm$ will be defined as quotient spaces of those defined in Section \ref{sec:pre}, see Section \ref{sec:quotient}, while function spaces of $G^{n-1}$ will be defined via trace theory, see Section \ref{sec:trace}.

\subsection{Quotient Spaces}\label{sec:quotient}
Define the extensions $e_0^\pm: \CRci(G_\pm)\to \CRci(G)$ via
\begin{align}\label{def:ext_op}
e_0^+\varphi(t,x',x_n):=
\begin{cases}
\varphi(t,x',x_n) & \text{if } x_n>0,\\
0 & \text{else,}
\end{cases}
\quad \text{and} \quad
e_0^-\varphi(t,x',x_n):=
\begin{cases}
\varphi(t,x',x_n) & \text{if } x_n<0,\\
0 & \text{else.}
\end{cases}
\end{align}
Then $e_0^\pm$ is continuous, and thus we may introduce the continuous restrictions $r^\pm:={}^t e_0^\pm:\mathcal{D}'(G)\to \mathcal{D}'(G_\pm)$.
Here $\mathcal{D}'(G)$ and $\mathcal{D}'(G_\pm)$ are the usual spaces of distributions, that is the strong dual of $\CRci(G)$ and $\CRci(G_\pm)$, respectively.
Thus, $\ker r^+$ is closed and so is $\ker r^+|_{E(G)}$ for any locally convex space $E(G)$ which is continuously embedded into $\cald'(G)$.
For the sake of readability, we will write $r^+$ instead of $r^+|_{E(G)}$ in the sequel.
\begin{definition}\label{def:restr_and_supp_space}
Let $E(G)$ be a locally convex space continuously embedded into $\mathcal{D}'(G)$.
Then we introduce the \emph{restricted spaces} $\overline{E}(G_\pm)$ and the \emph{supported spaces} $\dot E(\overline{G_\pm})$ via
\begin{align*}
\overline{E}(G_\pm):=r^\pm E(G) \quad \text{and} \quad \dot E(\overline{G_\pm}):=\{f\in E(G)\mid \supp f\subseteq \overline{G_\pm}\}.
\end{align*}
\end{definition}
As we mostly work with restricted and supported spaces on the upper half space, we sometimes write $\overline{E}:=\overline{E}(G_+)$ and $\dot E:=\dot E(\overline{G_+})$ if no confusion can arise.
  In particular, given any order function $\mu$, we obtain the restricted space
\begin{align*}
    \overline{H}^{\mu}_\perp(G_+)&:=\{f\in \cald'(G_+)\mid f=r^+ F \ \text{ for some } \ F\in H^{\mu}_\perp(G)\}, \\
    \|f\|_{\overline{H}^{\mu}_\perp}&:=\|f\|_{\mu, +}:=\inf\{\| F \|_{\mu} \mid F\in H^{\mu}_\perp(G),  r^+F=f\},
\end{align*}
  and the supported space
\begin{align*}
  \dot{H}^{\mu}_\perp(\overline{G_+})&:=\{f\in H^{\mu}_\perp(G)\mid \supp f\subseteq \overline{G_+}\},\\
    \|f\|_{\dot{H}^{\mu}_\perp}&:=\|f\|_{\mu}.
\end{align*}
It will be convenient to note that the restriction operator acts as intended on the supported spaces.
\begin{lemma}\label{lem:restr_supp_space}
It holds $r^{\pm}\dot{\cald}'(\overline{G_{\mp}})=\set{0}$.
\end{lemma}
\begin{proof}
    Let $f\in \dot{\cald}'(\overline{G_{\mp}})$, i.e. $\supp f\subseteq \overline{G_\mp}$.
    Then for $\varphi\in \cald(G_\pm)$ we have $\supp e_0^+\varphi\subseteq G_\pm$ and thus
    \begin{align*}
        \langle r^+f,\varphi\rangle&=\langle f,e_0^+\varphi\rangle=0.
        \qedhere
    \end{align*}
\end{proof}
\begin{remark}\label{rem:quot_space}
~\begin{enumerate}
\item\label{rem:quot_spacei} We use the traditional notation of H\"ormander for the restricted and supported spaces.
We note that \cite[Definition 1.63]{denk2013general} use slightly different terminology (besides the fact that in the present work we are dealing with time periodic functions).
In particular, our space $\overline{\cals}'(G_+)$ would be called ${}_0\cals'(G_+)$ in \cite{denk2013general}.
On the other hand, in their notation the space $\cals'(G_+)$ would correspond to $r^+\dot{\cals}'(\overline{G_+})$ in our notation:
Indeed, using the notation of \cite{denk2013general}, the inclusion $\cals'(G_+)\subseteq r^+\dot{\cals}'(\overline{G_+})$ follows from
\begin{align*}
\cals'(G_+)=r_0^+\cals'(G)=r^+(1-e^-r^-)\cals'(G)
\end{align*}
and the observation that $(1-e^-r^-)\cals'(G)\subseteq \dot{\cals}'(\overline{G_+})$, while the converse inclusion follows from $e_0^+r_0^+f=f$ for $f\in \dot\cals'(\overline{G_+})$ (see Lemma 1.65 in \cite{denk2013general}) and $r^+e^+_0 f=f$ for $f\in \cals'(G_+)$ in form of $r_0^+f=r^+e_0^+r_0^+f=r^+f$ for all $f\in \dot \cals'(\overline{G_+})$.
\item\label{rem:quot_spaceii} If $E(G)\subseteq \CRi(G)$, then for each $f\in \overline{E}(G_+)$ and each $j\in\N$ there is a unique extension of $\partial_{n}^j f$ to $\overline{G_+}$.
Therefore, we may regard $\overline{E}(G_+)$ as a space of functions defined on $\overline{G_+}$ instead of $G_+$, and suggestively write $E(\overline{G_+}):=\overline{E}(G_+)$.
\end{enumerate}
\end{remark}

\subsection{Trace Theory}\label{sec:trace}
The first step to treat the half space problem is to identify the trace spaces associated to our Newton polygon spaces. 
We will see that the appropriate spaces are again Newton polygon spaces defined on $G^{n-1}$, and that the corresponding polygons are obtained by 'translating' the polygon of the original space to the left.
A good understanding of those spaces will be essential to establish the {(non-)}existence of solutions in the later sections.
We consider the usual trace operators of $i$-th order: for $i,j\in \mathbb{N}$, we set
\begin{equation*}
    \Tr_j: \CRci(\overline{G_+})\to \CRci(G^{n-1}), \qquad \Tr_j\phi(t, x'):=\partial_{n}^j\phi(t, x', 0)
\end{equation*}
and
\begin{equation*}
    \Tr^{(i)}: \CRci(\overline{G_+})\to \CRci(G^{n-1})^{i}, \qquad \Tr^{(i)}\phi:=(\Tr_0\phi, \Tr_1\phi,..., \Tr_{i-1}\phi),
\end{equation*}
where we keep in mind the notation introduced in Remark \ref{rem:quot_space}.\ref{rem:quot_spaceii}.

\medskip

Now consider a Newton polygon $\mathcal{N}$, the associated order function $\mu=\mu(\mathcal{N})$, and its set of vertices $\mathcal{N}_V=\{(r_0, s_0),..., (r_{J+1}, s_{J+1})\}$.
We are looking for a function space $\mathbb{G}$ such that  the operator $\Tr^{(r_1)}$ defined above extends to a bounded and surjective operator 
\begin{equation*}
   \Tr^{(r_1)}: \overline{H}^{\mu}_\perp(G_+)\to \mathbb{G}.
\end{equation*}
\begin{corollary}\label{cor: interpolation_intersection}
     For all $r, s, r_0, s_0\geq 0$ such that $s_0\geq s$ ,  $r_0\geq r$ and $\theta\in (0, 1)$ ,
      \begin{equation*}
          [H^{(s, r)}_\perp (G), H^{(s, r_0)}_\perp(G))\cap H^{(s_0, r)}_\perp(G)]_{\theta}
          =H^{(s, (1-\theta)r+\theta r_0)}_\perp(G)\cap H^{((1-\theta)s+\theta s_0, r)}_\perp(G).
      \end{equation*}
\end{corollary}
\begin{proof}
	Consider the sets $E_0:=\set{(s,r)}$, $E_1:=\set{(s,r_0),(s_0,r)}$ and $E:=\set{(s,(1-\theta)r+\theta r_0),((1-\theta)s+\theta s_0,r)}$, as well as the associated Newton polygons $\caln(E_0)$, $\caln(E_1)$, and $\caln(E)$.
	For the associated order functions $\mu_0:=\mu(\caln(E_0))$, $\mu_1:=\mu(\caln(E_1))$ and $\mu:=\mu(\caln(E))$, we have by Proposition \ref{prop_potential_spaces} that
    \begin{align*}
        H^{\mu_0}_\perp(G)&=H^{(s, r)}_\perp (G), \\
        H^{\mu_1}(G)&=H^{(s, r_0)}_\perp(G))\cap H^{(s_0, r)}_\perp(G), \\
        H^{\mu}(G)&=H^{(s, (1-\theta)r+\theta r_0)}_\perp(G)\cap H^{((1-\theta)s+\theta s_0, r)}_\perp(G).
    \end{align*}
    Note that
    \begin{align*}
        \mu_0(\gamma)&=r+s\gamma,\\
        \mu_1(\gamma)&=\max\set{r_0+s\gamma,r+s_0\gamma},\\
        \mu(\gamma)&=\max\set{(1-\theta)r+\theta r_0+s\gamma,r+((1-\theta)s+\theta s_0)\gamma}.
    \end{align*}
	In particular, $(1-\theta)\mu_0+\theta\mu_1= \mu$, so that the assertion follows from Lemma \ref{lem: interp_L2}.
\end{proof}

 We now describe the trace space associated to Newton polygon spaces.
 As mentioned in the introduction, we will restrict ourselves to Newton polygons of a shape similar to the one of the CHG determinant symbol (\ref{eqn : CHG_det_polygon}). For integers $r_1, r_2$ with $r_1>r_2\geq 0$, and for  $s_2\geq 0 $ consider a Newton polygon $\mathcal{N}$ with vertices 
 \begin{equation*}
     \mathcal{N}_V:=\{(0,0), (r_1, 0), (r_2, s_2), (0, s_2)\}.
 \end{equation*}
 We set $\gamma^\perp_1 : =\frac{s_2}{r_2-r_1}\le 0$ the slope of the line segment joining $(r_2, s_2)$ to $(r_1, 0)$.
 The associated order function to $\mathcal{N}$ can be written as $\mu=(r_1-r_2)o_{-\gamma^\perp_1}+r_2$ with $\ord \mu=r_1$, where we recall the definition of the elementary order functions $o_y$ in the paragraph before Proposition \ref{prop: decomposition_elementary_order}.
 Order functions of such type will be called \emph{CHG-shaped}.
 For $j\in\set{0,\ldots, r_1-1}$, define the order functions
 \begin{align*}\label{eqn: def_trace_ord_fct}
      T_j(\mu) :=\left\{\begin{array}{ll}
      (r_1-r_2)o_{-\gamma^\perp_1} + r_2-j-\frac12,  & \text{for } j\in\set{0,..., r_2-1}, \\
     (r_1-j-\frac12)o_{-\gamma^\perp_1}, &\text{for } j\in\set{r_2,..., r_1-1}.
     \end{array}\right.
 \end{align*}
 Remark that the Newton polygon $T_j(\caln)$ associated to $T_j(\mu)$ is the Newton polygon $\mathcal{N}$ translated to the left by $j+\frac12$, cf. Figure \ref{fig: trace_spc}.
 In particular we have
 \begin{equation*}\label{eqn: def_trace_space}
    H^{T_j(\mu)}_\perp(G^{n-1})=\left\{\begin{array}{ll}
     H^{(s_2,r_2-j-\frac12)}_\perp(G^{n-1})
             \cap H^{(0,r_1-j-\frac12)}_\perp(G^{n-1})  & \text{for } j\in\set{0,..., r_2-1}, \\
     H^{(s_2+\gamma_1^\perp (j-r_2+\frac12),0)}_\perp (G^{n-1})   \cap H^{(0,r_1-j-\frac12)}_\perp(G^{n-1})&\text{for } j\in\set{r_2,..., r_1-1}.
     \end{array}\right.
 \end{equation*}
 We set 
 \begin{equation}\label{eqn: def_trace_space_vector}
     T^\mu_\perp(G^{n-1}):=\prod_{j=0}^{r_1-1}H^{T_j(\mu)}_\perp(G^{n-1}) ,
 \end{equation}
 and prove that it is the desired trace space.
\begin{theorem}\label{thm: trace}
    Let $\mu$ be a CHG-shaped order function and $r_1:=\ord\mu$.
    The trace operator $\Tr^{(r_1)}$ extends to $\Tr^{(r_1)}\in \call(H^{\mu}_\perp(G),T^\mu_\perp(G^{n-1}))$.
    Furthermore, it admits a right-inverse $E^{\mu}\in \call(T^\mu_\perp(G^{n-1}), H^{\mu}_\perp(G))$.
\end{theorem}
\begin{proof}
        For the sake of readability, we will write $\mathbb{E}:=\overline{H}^{\mu}_\perp(G_+)$, $\mathbb{G}:=T^\mu_\perp(G^{n-1})$, $H(s, r):= H_\perp^{(s, r)} (G^{n-1})$, and $\overline{H}(s,r) :=\overline{H}^{(s,r)}_\perp(G_+)$.
        Let us comment on the structure of the proof: In Step 1, we show the boundedness of the trace operator, while in Step 2 we construct a bounded right-inverse.
        Both steps are divided into two substeps: For each $j\in\set{0,\ldots,r_1-1}$, the shifted Newton polygon can have three or four vertices. We treat both cases for $j=0$ in the first substep. In the second substep we lift the argument from $j=0$ to general $j\in\set{0,\ldots,r_1-1}$.

        \medskip
        
        \textsc{Step} 1.
        We want to show that $\Tr^{(r_1)}$ extends to a well-defined and bounded operator $\mathbb{E}\to \mathbb{G}$.
        Obviously, it is sufficient to prove that $\Tr_j$ extends to a bounded operator $\mathbb{E}\to \mathbb{G}_j:=H^{T_j(\mu)}_\perp(G^{n-1})$ for all $j\in\set{0,..., r_1-1}$.

        \medskip
        
        \textsc{Step} 1.1.
        We start by showing that $\Tr_0$ extends to a well-defined and bounded operator $\mathbb{E}\to \mathbb{G}_0$. Here we have two cases to distinguish depending on the value of $r_2$.

        \medskip
        
         If $r_2=0$, then the Newton polygon associated to $\mathbb{E}$ is simply $\mathcal{N}(\{(r_1, 0), (0, s_2)\})$ and $\gamma^\perp_1=-\frac{s_2}{r_1}$. Thus an equivalent norm on $\mathbb{E}$ is given by 
         \begin{equation*}
             \|f\|_{\mathbb{E}} := \|\partial_{n}^{r_1}f\|_{L^2(\mathbb{R}_+, H(0,0))}+\|f\|_{L^2(\mathbb{R}_+, H(s_2, 0)\cap H(0, r_1))}.
         \end{equation*}
         By the trace method e.g. in Section 1.8 of \cite{triebel1995interpolation}, this implies the boundedness of 
         \begin{equation*}
             \Tr_0 : \mathbb{E}\to ( H(0, r_1)\cap H(s_2, 0),H(0,0) )_{\frac{1}{2r_1}, 2}= [ H(0, r_1)\cap H(s_2, 0),H(0,0) ]_{\frac{1}{2r_1}}.
         \end{equation*}
         Using Corollary \ref{cor: interpolation_intersection}, we find 
         \begin{align*}
             [H(0,r_1)\cap H(s_2, 0),H(0,0)]_{\frac{1}{2r_1}}
             &=H(0,r_1-\frac12)\cap H(s_1-\frac{\gamma_1^\perp}{2},0)=\mathbb{G}_0.
         \end{align*}
         
         If $r_2\geq 1$, then an equivalent norm on $\overline{H}(s_2, r_2)$ is given by 
         \begin{equation*}
             \|f\|_{\overline{H}(s_2, r_2)}\simeq  \|\partial_n^{r_2}f\|_{L^2(\mathbb{R}_+, H(s_2, 0))}+\|f\|_{L^2(\mathbb{R}_+, H(s_2, r_2))}.
         \end{equation*}
         Using the trace method as above, we get the boundedness of 
         \begin{equation*}
             \Tr_0: \overline{H}(s_2, r_2)\to H(s_2,r_2-\frac12).
         \end{equation*}
         Substituting $(s_2, r_2)$ by $(0,r_1)$ in the above, we also get boundedness of 
         \begin{equation*}
             \Tr_0 : \overline{H}(0, r_1)\to H(0,r_1-\frac12),
         \end{equation*}
         and thus 
         \begin{align*}
             \Tr_0 : \mathbb{E}&=\overline{H }(s_2, r_2)\cap \overline{H}(0, r_1)
             \to H(s_2,r_2-\frac12)\cap H(0,r_1-\frac12)=\mathbb{G}_0
         \end{align*}
         is also bounded.

         \medskip
         
         \textsc{Step} 1.2.
         For all $j=0,..., r_1-1$, remark that $\Tr_j =\Tr_0 \circ \partial_n^j$ and define 
         \begin{equation*}
             \mathbb{E}_j:=\left\{\begin{array}{ll}
             \overline{H}(s_2, r_2-j)\cap \overline{H}(0, r_1-j)    & \text{if }j=0,..., r_2-1  \\
             \overline{H}(s_2+\gamma_1^\perp(j-r_2), 0) \cap    \overline{H}(0, r_1-j)  & \text{if } j= r_2,..., r_1-1
             \end{array}\right.
         \end{equation*}
         Remark that the Newton polygon associated to $\mathbb{E}_j$ is the Newton polygon of the space $\mathbb{E}$ translated to the left by $j$, and that $\partial_n^j: \mathbb{E}\to \mathbb{E}_j$ is bounded. Furthermore, replacing $\mathbb{E}, \mathbb{G}_0$ by $\mathbb{E}_j, \mathbb{G}_j$ in STEP 1.1, one gets that $\Tr_0: \mathbb{E}_j\to \mathbb{G}_j$ is bounded, and the conclusion follows.

          \medskip
          
          \textsc{Step} 2.
          We now construct a bounded right-inverse. This part of the proof is adapted from \cite{denk2008inhomogeneous}. Pick $\sigma=(\sigma_0, ..., \sigma_{r_1-1}) \in \mathbb{G}$. For $j=0,..., r_1-1$, we define the operator 
          \begin{equation*}
              A_j :=\left\{\begin{array}{ll}
             (1-\Delta')^{\frac{1}{2}}      & \text{if } j\in\set{0,..., r_2-1} \\
              (1-\Delta')^{\frac{1}{2}}+(1+\partial_t)^{-\gamma^\perp_1}      & \text{if } j\in\set{r_2,..., r_1-1},
              \end{array}\right. 
          \end{equation*}
          where the Laplace operator $\Delta'$ acts in the variables $x'\in \R^{n-1}$.
          These operators generate analytic semigroups in all considered Newton polygon spaces (in fact, they admit a bounded $H^\infty$-calculus by a Fourier multiplier argument as in \cite[Example 10.2]{KuW04}), and in particular they admit maximal $L^2$ regularity.
          Define
        \begin{equation*}
            \eta_j(x_n):=\sum_{k=0}^{r_1-1} c_{kj}\mathrm{e}^{-kx_nA_j}(A^{-j}_j\sigma_j),\qquad j\in\set{0,..., r_1-1},
        \end{equation*}
        where we can choose the coefficients $c_{kj}$ in such a way that for $j,m\in\set{0,...,r_1-1}$ it holds
        \begin{equation}\label{eqn: choice_coeff}
            \partial_{n}^{m}\eta_j(0)=\delta_{m, j}\sigma_j.
        \end{equation}
        Indeed, we have 
        \begin{align*}
             \partial_{n}^{m}\eta_j(0)
             &=\sum_{k=0}^{r_1-1}(-k)^{m}c_{kj}A^{m-j}_{j}\sigma_j.
        \end{align*}
        Therefore, (\ref{eqn: choice_coeff}) is verified if $VC_j=e_j$ for $j\in\set{1,...,r_1}$, where $C_j=(c_{1j},...,c_{r_1j})$ and $V$ is the Vandermonde matrix. Since $V$ is invertible, we can set $C_j=V^{-1}e_j$ to achieve this.
        We now set 
        \begin{equation*}
            \eta=\sum_{j=0}^{r_1-1}\eta_j.
        \end{equation*}
        By (\ref{eqn: choice_coeff}), we have $\partial_{n}^{m}\eta=\sigma_m$ for $m\in\set{0,...,r_1-1}$ and thus $\Tr^{(r_1)}\eta=\sigma$.
        It is left to show that $\eta_j\in \mathbb{E}$ for all $j\in \{0,...,r_1-1\}$.
        Once again, we will proceed in two substeps.

        \medskip
        
        \textsc{Step} 2.1.
        We show that $\eta_0 \in \mathbb{E}$  and  $\|\eta_0\|_{\mathbb{E}}\lesssim \|\sigma_0\|_{\mathbb{G}_0}$, which due to
        \begin{align}\label{eqn: trace_decomp}
            \|\eta_0\|_{\mathbb{E}}\simeq \|\eta_0\|_{L^2(\R_+;H(s_2, r_2))} + \|\partial_n^{r_2}\eta_0\|_{L^2(\R_+;H(s_2, 0))}  +\|\eta_0\|_{L^2(\mathbb{R}_+ ; H(0, r_1))} + \|\partial_n^{r_1}\eta_0 \|_{L^2(G_+)}
        \end{align}
        comes down to estimating the four contributions to the right-hand side individually.
        Once again, we have two cases to consider depending on the value of $r_2$. 

        \medskip
        
        We first show the case $r_2\geq 1$. In this case we have $A_0=(1-\Delta')^{\frac{1}{2}}$.
        For $j\in\set{1,2}$ consider $A_0$ acting in $X:=H( s_j, r_j-1)$ with domain $D(A_0):=H(s_j, r_j)$. We claim that 
            \begin{equation*}
                \mathbb{G}_0 \hookrightarrow [X, D(A_0)]_{\frac{1}{2}}.
            \end{equation*}
            Indeed, Corollary \ref{cor: interpolation_intersection} directly shows that the right-hand side is equal to $H(s_j, r_j-\frac{1}{2})$, which is a vertex of the Newton polygon associated to the space $\mathbb{G}_0$.
            Therefore, Lemma \ref{lemma: point_inside} yields the claim.
            The maximal regularity of $A_0$ now shows 
            \begin{equation*}
                \eta_0 \in \overline{H}^1(\mathbb{R}_+ ; H(s_j, r_j-1)) \cap L^2(\mathbb{R}_+; H(s_j, r_j)) 
            \end{equation*}
            and in particular $\|\eta_0\|_{L^2 (\mathbb{R}_+ ; H(s_j, r_j))} \lesssim \|\sigma_0\|_{\mathbb{G}_0}$.
            This is the desired estimate for the first and the third contribution to the right-hand side of \eqref{eqn: trace_decomp}.
            For the second and fourth contribution to the right-hand side of \eqref{eqn: trace_decomp}, remark that 
            \begin{align*}
            \partial_n^{r_j-1} \eta_0=\sum_{k=0}^{r_1-1} (-k)^{r_j-1}c_{k0}\mathrm{e}^{-kx_nA_0}(A_0^{r_j-1}\sigma_0), \qquad j\in\set{1,2}.   \end{align*}
            Set $X:=H(0,0)$ and $D(A_0)=H(0,1)$.  Here one has
        \begin{equation*}
            [X, D(A_0)]_{\frac{1}{2}}= H(0, \frac{1}{2}),
        \end{equation*}
        and since $A_0$ is $\mathcal{N}$-elliptic with weight function $w(\tau, \xi)\simeq 1+\vert \xi\vert$, Theorem \ref{thm: whole space} shows
        \begin{equation*}
            A_0^{r_j-1}\in \call(\mathbb{G}_0,
        H(0, \frac{1}{2})\cap H(-\frac{\gamma_1^\perp}{2}, 0))\hookrightarrow \call(\mathbb{G}_0, H(0, \frac{1}{2})). 
        \end{equation*}
        Now the maximal regularity of $A_0$ yields $\partial_n^{r_j-1}\eta_0 \in \overline{H}^1(\mathbb{R}_+; H(0,0))$ and $\|\partial_n^{r_j}\eta_0\|_{L^2(G_+)}\lesssim \|\sigma_0\|_{\mathbb{G}_0}$. 
        
        \medskip
        
        For the case $r_2=0$, the first two contributions to the right-hand side of \eqref{eqn: trace_decomp} collapse to $\|\eta_0\|_{\overline{H}(s_2,0)}$. Observe that now $A_0= (1-\Delta')^{\frac{1}{2}}+(1+\partial_t)^{-\gamma_1^\perp}$.
        We consider $A_0$ acting in $X:=H(s_2+\gamma_1^\perp , 0)$ with domain $D(A_0):= H(s_2, 0)\cap H(s_2+\gamma^\perp_1, 1)$. We claim that 
        \begin{equation*}
            \mathbb{G}_0 \hookrightarrow [X, D(A_0)]_{\frac{1}{2}}. 
        \end{equation*}
        Indeed, by Corollary \ref{cor: interpolation_intersection}, the right-hand side is equal to $H(s_2+\frac{\gamma_1^\perp}{2}, 0)\cap H(s_2+\gamma_1^\perp, \frac{1}{2})$, and both the points $(s_2+\frac{\gamma_1^\perp}{2}, 0)$ and $(s_2+\gamma_1^\perp, \frac{1}{2})$ belong to the Newton polygon associated to the space $\mathbb{G}_0$, so that the embedding follows by Lemma \ref{lemma: point_inside}. 
        The maximal regularity of $A_0$ thus yields 
        \begin{equation*}
            \eta_0 \in \overline{H}^1(\mathbb{R}_+; H(s_2+\gamma_1^\perp , 0))\cap L^2(\mathbb{R}^+;H(s_2, 0)\cap H(s_2+\gamma^\perp_1, 1))\hookrightarrow \overline{H}(s_2, 0)
        \end{equation*}
        with $\|\eta_0\|_{\overline{H}(s_2, 0)} \lesssim\|\sigma_0\|_{\mathbb{G}_0}$. 

        \medskip
        
        For the third contribution to the right-hand side of \eqref{eqn: trace_decomp} we run a similar argument applied to $X:=H(0, r_1-1)$ and $D(A_0):= H(-\gamma^\perp_1, r_1-1)\cap H(0, r_1)$.
        This yields $\eta_0\in L^2(\mathbb{R}_+, H(0, r_1))$ together with the estimate $\|\eta_0\|_{L^2(\mathbb{R}_+, H(0, r_1))}\lesssim \|\sigma_0\|_{\mathbb{G}_0}$.

        \medskip
        
        For the fourth contribution to the right-hand side of \eqref{eqn: trace_decomp} we use a similar argument to the one used in the case $r_2\geq 1$, here taking $X:=(0,0)$ and $D(A_0)=H(\gamma_1^\perp, 0)\cap H(0, 1)$ to get  $\partial_n ^{r_1-1}\eta_0 \in \overline{H}^1(\mathbb{R}_+; H(0,0))$ with the desired estimate $\|\partial_n^{r_1} \eta_0\|_{L^2(G_+)}\lesssim \|\sigma_0\|_{\mathbb{G}_0}$.

        \medskip
        
        \textsc{Step} 2.2
            We now show that $\eta_j\in \mathbb{E}$ for all $j\in\set{0,...,r_1-1}$ and $\|\eta_j\|_{\mathbb{E}}\lesssim \|\sigma_j\|_{\mathbb{G}_j}$. 
            We set
            \begin{equation*}
                \mathbb{X}_j :=\left\{\begin{array}{ll}
                H(s_2, r_2-\frac{1}{2})\cap H(0, r_1-\frac{1}{2})    & \text{if } j\in\set{0,..., r_2-1}  \\
                H(-\gamma^\perp_1 (r_1-\frac{1}{2}), 0) \cap H(0, r_1-\frac{1}{2})     & \text{if } j\in\set{r_2,..., r_1-1}.
                \end{array}\right. 
            \end{equation*}
            As in the previous step we obtain by Theorem \ref{thm: whole space} that $A_j^{-j}\in \call_{\mathrm{iso}}(\mathbb{G}_j,\mathbb{X}_j)$.
            Therefore, it makes sense to define the maximal regularity spaces $\mathbb{Y}_j$ of $A_j$ corresponding to 'initial data' in $\mathbb{X}_j$, i.e.
            \begin{equation*}
                \mathbb{Y}_j :=\left\{\begin{array}{ll}
                \mathbb{E}     & \text{if } j\in\set{0,..., r_2-1},  \\
                \overline{H}(-\gamma^\perp_1 r_1, 0)\cap \overline{H}(0, r_1)     & \text{if } j\in\set{r_2,..., r_1-1}. 
                \end{array}\right.
            \end{equation*}
            Replacing $\mathbb{E}, \mathbb{G}_0, \sigma_0$ in STEP 2.1 by $\mathbb{Y}_j, \mathbb{X}_j, A_j^{-j}\sigma_j$, we infer that 
            $\|\eta_j\|_{\mathbb{Y}_j} \lesssim \|A_j^{-j}\sigma_j\|_{\mathbb{X}_j} \lesssim \|\sigma_j\|_{\mathbb{G}_j}$. 
            To conclude, we remark that the Newton polygon associated to the space $\mathbb{E}$ is included in $\mathbb{Y}_j$ for all $j\in\set{0,..., r_1-1}$. Thus $\mathbb{Y}_j \hookrightarrow \mathbb{E}$ by Lemma \ref{lemma: point_inside}, which concludes the proof.
            \qedhere
\end{proof}
In conclusion, the trace spaces associated to $\overline{H}^{\mu}_\perp(G_+)$ can be described by a Newton polygon.
 We particularise these results to the Cahn-Hilliard-Gurtin determinant.
\begin{example}\label{ex: trace_CHG}
    We apply the results of Theorem \ref{thm: trace} to the space $H^{\mu_D}_\perp(G_+)$, where $\mu_D$ is the order function associated to the Cahn-Hilliard-Gurtin determinant (\ref{eqn: determinant}).
    By Example \ref{ex: CHD_order_fct} we have $\mu_D=2o_\frac12+2$.
    Therefore, the order functions associated to the trace spaces are successively given by
    \begin{align*}
        T_0(\mu_D)=2o_{\frac{1}{2}} + \frac32, 
        \quad T_1(\mu_D)= 2o_{\frac{1}{2}} + \frac{1}{2},
        \quad T_2(\mu_D)= \frac32o_{\frac{1}{2}},
        \quad T_3(\mu_D)= \frac12o_{\frac{1}{2}}. 
    \end{align*}
    In other words, the Newton polygons associated to the trace spaces are given by
    \begin{align*}
        &T_0(\mathcal{N})_V =\{(0, 0), (\frac{7}{2}, 0), (\frac{3}{2}, 1), (0, 1)\}\\
        &T_1(\mathcal{N})_V =\{(0, 0), (\frac{5}{2}, 0), (\frac{1}{2}, 1), (0, 1)\}\\
        &T_2(\mathcal{N})_V =\{(0, 0), (\frac{3}{2}, 0), (0, \frac34)\}\\
        &T_3(\mathcal{N})_V =\{(0, 0), (\frac{1}{2}, 0), (0, \frac14)\}.
    \end{align*}
    which are indeed the translates of the Newton polygon of $D$, which we recall is described by 
    \begin{equation*}
        \mathcal{N}_V(D)=\{(0,0), (4, 0), (2, 1), (0,1)\}.
    \end{equation*}
\end{example}
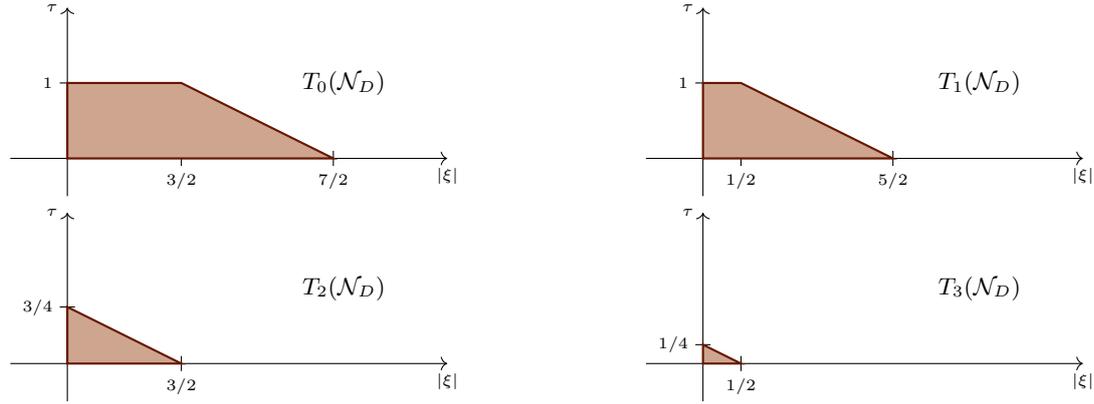
\begin{figure}[h]
   
   \begin{subfigure}[b]{.49\textwidth}
       \begin{tikzpicture}
 \draw[->] (-.75,0) -- (5,0) node[below , font=\tiny] {$\vert \xi\vert$};
  \draw[->] (0,-.5) -- (0,2) node[left, font=\tiny] {$\tau$}; 
  
   \draw (1.5,0.1) -- (1.5,-0.1);
  \draw (3.5,0.1) -- (3.5,-0.1);
    \node[below, yshift= -2pt, font=\tiny] at (1.5,0) {$3/2$};
    \node[below, yshift=-2pt, font=\tiny] at (3.5, 0) {$7/2$};

    \draw (-0.1, 1) -- (0.1, 1);
    \node[left, xshift=-2pt, font= \tiny] at (0,1) {$1$};
  \coordinate (A) at (0,0);
  \coordinate (B) at (3.5,0);
  \coordinate (C) at (1.5,1);
  \coordinate (D) at (0,1);

  \filldraw[color=Sepia, fill=Sepia!30,thick] (A)--(B)--(C)--(D)--cycle;

     
      \node[xshift=4pt, font=\small] at (7/2, 1) {$T_0(\mathcal{N}_D)$};
\end{tikzpicture}
   \end{subfigure}
   \hfill
   \begin{subfigure}[b]{.49\textwidth}
       \begin{tikzpicture}
 \draw[->] (-.75,0) -- (5,0) node[below , font=\tiny] {$\vert \xi\vert$};
  \draw[->] (0,-.5) -- (0,2) node[left, font=\tiny] {$\tau$}; 
  
   \draw (0.5,0.1) -- (0.5,-0.1);
  \draw (2.5,0.1) -- (2.5,-0.1);
    \node[below, yshift= -2pt, font=\tiny] at (0.5,0) {$1/2$};
    \node[below, yshift=-2pt, font=\tiny] at (2.5, 0) {$5/2$};

    \draw (-0.1, 1) -- (0.1, 1);
    \node[left, xshift=-2pt, font= \tiny] at (0,1) {$1$};
  \coordinate (A) at (0,0);
  \coordinate (B) at (2.5,0);
  \coordinate (C) at (0.5,1);
  \coordinate (D) at (0,1);

  \filldraw[color=Sepia, fill=Sepia!30,thick] (A)--(B)--(C)--(D)--cycle;

     
      \node[xshift=4pt, font=\small] at (7/2, 1) {$T_1(\mathcal{N}_D)$};
\end{tikzpicture}
\end{subfigure}

   \begin{subfigure}{.49\textwidth}
      
      \begin{tikzpicture}
 \draw[->] (-.75,0) -- (5,0) node[below , font=\tiny] {$\vert \xi\vert$};
  \draw[->] (0,-.5) -- (0,2) node[left, font=\tiny] {$\tau$}; 
  
   \draw (1.5,0.1) -- (1.5,-0.1);
  
    \node[below, yshift= -2pt, font=\tiny] at (1.5,0) {$3/2$};

    \draw (-0.1, .75) -- (0.1, .75);
    \node[left, xshift=-2pt, font= \tiny] at (0,0.75) {$3/4$};
  \coordinate (A) at (0,0);
  \coordinate (B) at (1.5,0);
  \coordinate (C) at (0, 0.75);

  \filldraw[color=Sepia, fill=Sepia!30,thick] (A)--(B)--(C)--cycle;

     
      \node[xshift=4pt, font=\small] at (7/2, 1) {$T_2(\mathcal{N}_D)$};
\end{tikzpicture}
   \end{subfigure}\hfill
   \begin{subfigure}{.49\textwidth}
       \begin{tikzpicture}
 \draw[->] (-.75,0) -- (5,0) node[below , font=\tiny] {$\vert \xi\vert$};
  \draw[->] (0,-.5) -- (0,2) node[left, font=\tiny] {$\tau$}; 
  
   \draw (0.5,0.1) -- (0.5,-0.1);
  
    \node[below, yshift= -2pt, font=\tiny] at (0.5,0) {$1/2$};

    \draw (-0.1, 0.25) -- (0.1, 0.25);
    \node[left, xshift=-2pt, font= \tiny] at (0,0.25) {$1/4$};
  \coordinate (A) at (0,0);
  \coordinate (B) at (0.5,0);
  \coordinate (C) at (0, 0.25);

  \filldraw[color=Sepia, fill=Sepia!30,thick] (A)--(B)--(C)--cycle;

     
      \node[xshift=4pt, font=\small] at (7/2, 1) {$T_3(\mathcal{N}_D)$};
\end{tikzpicture}
   \end{subfigure}
   \caption{The Newton polygons associated to the trace spaces of $\mu_D$.}
   \label{fig: trace_spc}
\end{figure}

\subsection{Operators Preserving Upper and Lower Support}
In order to adapt the results of Section 2 to the half space problem, the main issue we are facing is that pseudodifferential operators are \emph{a priori} nonlocal.
Thus, not every Fourier multiplier $\op[m]: E(G)\to \mathcal{S}'(G)$ induces a well-defined operator $\overline{E}(G_+)\to \overline{\cals}{}'(G_+)$. 
\begin{definition}[support-preserving operators]\label{def: supp_pres_op}
    Let $E(G),F(G)\hookrightarrow \cald'(G)$ be locally convex spaces and $\opA\in \call(E(G),F(G))$.
    \begin{enumerate}
    \item $\opA$ \emph{preserves support in $\overline{G_\pm}$} (or \emph{preserves upper/lower support}) if for every $f\in \dot E(\overline{G_\pm})$, one has $\opA f\in \dot F(\overline{G_\pm})$.
    \item If $\opA$ preserves support in $\overline{G_\pm}$, the induced operators $\dot{\opA}_{\pm}:r^\pm\dot E(\overline{G_\pm})\to r^\pm\dot F(\overline{G_\pm})$ and $\opA_\mp : \overline{E}(G_\mp)\to \overline{F}(G_\mp)$ are defined via
    \begin{align*}
     \dot{\opA}_\pm u:=r^\pm Ae_0^\pm u \qquad \text{and} \qquad \opA_\mp u:=r^\mp \opA U,
    \end{align*}
    where $U\in E(G)$ is any extension of $u$, that is, $r^\mp U=u$.
    \end{enumerate}
\end{definition}
Observe that $\opA_\mp$ is well-defined, i.e., the above definition does not depend on the choice of the extension $U$.
Indeed, let $U_1, U_2 \in E(G)$ with $r^\mp U_1=r^\mp U_2=u$.
Then $U_1-U_2\in \dot{E}(\overline{G_\pm})$.
Since $\opA$ preserves support in $\overline{G_\pm}$, it holds $\opA U_1-\opA U_2\in \dot{F}(\overline{G}_\pm)$, which in turn implies $r^\mp\opA U_1=r^\mp\opA U_2$ by Lemma \ref{lem:restr_supp_space}.
Moreover, $\dot\opA_\pm$ is the restriction of $\opA_\pm$ to $r^\pm \dot E(\overline{G_\pm})$ if both exist.
Finally, if $\opA$ is a local operator (for example a differential operator), then $\opA:\overline{E}(G_\mp)\to \overline{F}(G_\mp)$ has a canonical meaning, and it holds $\opA=\opA_\mp$. 
\begin{proposition}\label{prop: boundedness_lower_support}
    Let $E(G), F(G),H(G)\hookrightarrow \cald'(G)$ be normed spaces and let $\opA\in \call(E(G),F(G))$, $\opB\in \call(F(G),H(G))$ preserve lower support.
    \begin{enumerate}
        \item It holds $\dot\opA_-\in \call(r^-\dot{E}(\overline{G_-}),r^-\dot{F}(\overline{G_-}))$ and $\opA_+\in \call(\overline{E}(G_+),\overline{F}(G_+))$, and both operator norms do not exceed $\|\opA\|_{\call(E(G),F(G))}$.
        \item $\opB\opA\in \call(E(G),H(G))$ preserves lower support, and $(\opB\opA)_-^{\boldsymbol{\cdot}}=\dot\opB_-\dot\opA_-$ and $(\opB\opA)_+=\opB_+\opA_+$.
        \item If $\opA\in \call_{\mathrm{iso}}(E(G),F(G))$ and $\opA^{-1}$ preserves lower support, then $\dot\opA_-\in \call_{\mathrm{iso}}(r^-\dot{E}(\overline{G_-}),r^-\dot{F}(\overline{G_-}))$ and $\opA_+\in \call_{\mathrm{iso}}(\overline{E}(G_+),\overline{F}(G_+))$ with $(\dot\opA_-)^{-1}=(\opA^{-1})_-^{\boldsymbol{\cdot}}=:\dot{\opA}_-^{-1}$ and $(\opA_+)^{-1}=(\opA^{-1})_+:=\opA^{-1}_+$.
    \end{enumerate}
\end{proposition}
\begin{proof}
\begin{enumerate}
    \item The assertion for $\dot\opA_-$ is trivial.
    For the assertion about $\opA_+$, let $u\in \overline{E}(G_+)$.  For any $U\in E(G)$ such that $r^+U=u$, it holds $\opA_+u=r^+ \opA U$, and thus
    \begin{align*}
        \|\opA_+ u\|_{\overline{F}(G_+)}&=\inf\{ \|V\|_{F(G)} \mid V\in F(G) , r^+V=\opA_+u\}\\
        &\leq \|\opA U\|_{F(G)}\leq \|\opA\| _{\mathcal{L}(E(G), F(G))} \|U\|_{E(G)}.
    \end{align*}
    Since the estimate holds for every extension $U\in E$ of $u$, we conclude
    \begin{equation*}
        \|\opA_+u\|_{\overline{F}(G_+)}
        \leq \|\opA\|_{\mathcal{L}(E(G), F(G))}\inf\{ \|U\|_{E(G)} \mid U\in E(G), r^+U=u\}
        = \|\opA\| _{\mathcal{L}(E, F)}\|u\|_{\overline{E}(G_+)}.
    \end{equation*}
    \item By definition, $\opB\opA\in \call(E(G),H(G))$ preserves lower support.
    By the previous point, we thus have $(\opB\opA)_+\in \call(\overline{E}(G_+),\overline{H}(G_+))$, $\opA_+\in \call(\overline{E}(G_+),\overline{F}(G_+))$, and $\opB_+\in \call(\overline{F}(G_+),\overline{H}(G_+))$.
    Moreover, for $u\in \overline{E}(G_+)$ and $U\in E(G)$ with $r^+U=u$ we have
    \begin{align*}
        \opB_+\opA_+u=\opB_+r^+\opA U=r^+\opB\opA U=(\opB\opA)_+u,
    \end{align*}
    where we have used that $r^+V=v$ for $V:=\opA U\in F(G)$ and $v:=r^+\opA U\in \overline{F}(G_+)$.
    \item Follows directly from the previous point with $\opB:=\opA^{-1}$.
    \qedhere
\end{enumerate}
\end{proof}

In view of Proposition \ref{prop: boundedness_lower_support}, if $\opA\in \call_{\mathrm{iso}}(H^{\mu}_\perp(G),L^2_\perp(G))$ preserves lower support, then one can consider the equation $\opA_+u=f$ on the half space $G_+$.
If $\opA^{-1}$ preserves lower support as well, then this equation is uniquely solvable, the solution being directly given by $u:=\opA^{-1}_+ f$.
However, the situation gets more complicated when $\opA^{-1}$ does not possess this property, since there is no direct way to define $\opA^{-1}_+$ as an inverse of $\opA_+$.
As a first observation, we show that in the case where both $\opA^{-1}$ and $\opA$ preserve upper support, there is a unique solution in the Newton polygon spaces by the trace theory developed in Section 3 if additional boundary conditions are prescribed. 
We first characterize the elements of $r^+\dot{H}^{\mu}_\perp(\overline{G_+})$ by their trace. 
\begin{lemma}\label{lemma: zero_trace}
    Let $\mu$ be a CHG-shaped order function and $r_1:=\ord\mu$.
    Then 
    \begin{equation*}
        r^+\dot{H}^{\mu}_\perp(\overline{G_+})=\{u\in \overline{H}^{\mu}_\perp(G_+): \Tr^{(r_1)}u=0\}. 
    \end{equation*}
    Moreover, $r^+\dot{H}^{0}_\perp(\overline{G_+})=\overline{H}^{0}_\perp(G_+)=L^2_\perp(G_+)$.
\end{lemma}
\begin{proof}
    As usual, we denote  $\mathcal{N}$ the Newton polygon of $\mu$ and $\mathcal{N}_V=\{(r_0, s_0),..., (r_{3}, s_{3})\}$ the set of its vertices.
    \\
    Consider first $u\in  \overline{H}^{\mu}_\perp(G_+)$ with $\Tr^{(N)} u=0$.
    Then $u\in \overline{H}^{r_j}(\mathbb{R}^n_+; H^{s_j}_\perp(\mathbb{T}))$ with $\Tr^{(r_j)}=0$ for all $j\in\set{0,1,2,3}$.
    By \cite[Theorem VIII.1.6.8.]{amann2019linear}, there is $\seqkN{u}\subseteq\CRci(\R^n_+;H^{s_j}_\perp(\mathbb{T}))$ such that $u_n\to u$ in $\overline{H}^{r_j}(\mathbb{R}^n_+; H^{s_j}_\perp(\mathbb{T}))$.
    In particular, $e_0^+u_n$ converges in $H^{r_j}(\mathbb{R}^n; H^{s_j}_\perp(\mathbb{T}))=H^{(r_j,s_j)}_\perp(G)$, and the limit is $U:=e_0^+u$.
    Therefore, $U\in \bigcap_{j=0}^{J+1} H^{(r_j,s_j)}_\perp(G)=H^{\mu}_\perp(G)$ and $\supp U\subseteq \overline{G_+}$, that is $U\in \dot H^{\mu}_\perp(\overline{G_+})$.
    Thus $u=r^+U\in r^+\dot{H}^{\mu}_\perp(\overline{G_+})$. 

\medskip

    Conversely
    let $U\in \dot{H}^{\mu}_\perp(G)$ and $u:=r^+U$.
    Then $u\in \overline{H}^{\mu}_\perp(G_+)$ by definition.
    Moreover $\Tr^{(r_1)}u=\Tr^{(r_1)}v= 0$, where $v:=r^+ V=0$ and $V(t,x',x_n):=U(t,x',-x_n)$.
\end{proof}
We recall the trace space $T^\mu_\perp(G^{n-1})$ and right-inverse $E^{\mu}\in \call(T^\mu_\perp(G^{n-1}),\overline{H}^{\mu}(G_+))$ from Theorem \ref{thm: trace}.
\begin{proposition}\label{prop: upper_support_preserving}
    Let $\mu$ and $\nu$ be strictly positive order functions such that $\mu+\nu$ is CHG-shaped, and set $r_1:=\ord (\mu+\nu)$.
    Moreover, suppose that $\opA\in \call_{\mathrm{iso}}(H^{\mu+\nu}_\perp(G),H^{\nu}_\perp(G))$ is such that
    \begin{align*}
        \text{$\opA$ preserves upper and lower support, and $\opA^{-1}$ preserves upper support.}
    \end{align*}
    Then for all $f\in r^+\dot H^{\nu}_\perp(G_+)$ and $g\in T^{\mu+\nu}_\perp(G^{n-1})$ with $\opA_+E^{\mu+\nu}g\in r^+\dot H^{\nu}_\perp(G_+)$, the problem 
    \begin{align*}
        \left\{\begin{array}{rcl}
        \opA_+u&=&f       \\
        \Tr^{(r_1)}u&=&g    
        \end{array}\right.
    \end{align*}
    admits a unique solution $u\in \overline{H}^{\mu+\nu}_\perp(G_+)$ given by the formula $u=v+w$ with
    \begin{align*}
        v=E^{\mu+\nu}g\in \overline{H}^{\mu+\nu}_\perp(G_+), \quad  w:=\dot\opA_+^{-1}(f-\opA_+E^{\mu+\nu}g)\in r^+\dot H^{\mu+\nu}_\perp(\overline{G_+}).
    \end{align*}
    Furthermore, 
    \begin{equation*}
        \|u\|_{\mu+\nu, +}\lesssim_{\mu,\nu,\opA} \|f\|_{\nu,+} + \|g\|_{\mathbb{G}}.
    \end{equation*}
\end{proposition}
\begin{proof}
    By Theorem \ref{thm: trace} and Proposition \ref{prop: boundedness_lower_support} it holds $v\in \overline{H}^{\mu+\nu}_\perp(G_+)$ and $w\in r^+\dot H^{\mu+\nu}_\perp(\overline{G_+})$ together with the claimed estimate.
    Since $\opA_+\dot\opA_+^{-1}=\id$ on $\dot H^{\nu}(\overline{G_+})$, we also obtain $\opA_+u=f$.
    Moreover, since $\Tr^{(r_1)}$ vanishes on $r^+\dot H^{\mu+\nu}_\perp(\overline{G_+})$ by Lemma \ref{lemma: zero_trace}, we have $\Tr^{(r_1)}u=\Tr^{(r_1)}v=g$. 

    \medskip
    
    We now prove the uniqueness.
    Assume that $u\in \overline{H}^{\mu+\nu}_\perp(G_+)$ is a solution to the problem for $f=0$ and $g=0$. Then $u\in r^+\dot H^{\mu+\nu}_\perp(\overline{G_+})$ by Lemma \ref{lemma: zero_trace}, and thus $0=\opA_+u=\dot\opA_+u$. Hence $u=\dot\opA_+^{-1}\dot\opA_+u=0$.
\end{proof}
Most differential operators of interest, and in particular the CHG-determinant operator, fulfill neither the hypothesis of Proposition \ref{prop: boundedness_lower_support} nor those of Proposition \ref{prop: upper_support_preserving}.
However, they can often be decomposed as a product of operators, one verifying the hypothesis of Proposition \ref{prop: boundedness_lower_support} and the other verifying the hypothesis of Proposition \ref{prop: upper_support_preserving}. This motivates the following class of admissible operators.
\begin{definition}\label{def: admissible}
    Let $\nu$ and $\mu$ be strictly positive order functions.
    Then $\opA\in\call_{\textrm{iso}}(H^{\mu+\nu}_\perp(G),H^{\nu}_\perp(G))$ is \emph{admissible} if there exist stictly positive order functions $\mu_+$ and $\mu_-$ with $\mu=\mu_+ +\mu_-$ and $\opA^-\in\call_{\textrm{iso}}(H^{\mu_-+\nu}_\perp(G),H^{\nu}_\perp(G))$, $\opA^+\in\call_{\textrm{iso}}(H^{\mu+\nu}_\perp(G),H^{\mu_-+\nu}_\perp(G))$ such that the following conditions are verified. 
    \begin{enumerate}
        \item $\opA=\opA^-\opA^+$.
        \item Both $\opA^+$ and $\opA^-$ preserve lower support.
        \item Both $\opA^+$ and $(\opA^+)^{-1}$ preserve upper support. 
        \item $(\opA^-)^{-1}$ preserves lower support.
    \end{enumerate}
\end{definition}
By Proposition \ref{prop: boundedness_lower_support}, $\opA$ preserves lower support and $\opA_+=\opA^-_+\opA^+_+\in \call(\overline{H}^{\mu+\nu}(G_+),\overline{H}^{\nu}(G_+))$, and it holds $\opA^-_+\in\call_{\mathrm{iso}}(\overline{H}^{\mu_-+\nu}_\perp(G_+), \overline{H}^{\nu}_\perp(G_+))$ with the inverse given by $(\opA^-)^{-1}_+$.
Observe that $\opA^+_+$ is not element of $\call_{\mathrm{iso}}(\overline{H}^{\mu+\nu}(G_-), \overline{H}^{\mu_-+\nu}(G_-))$ in general, since its inverse is not defined.\footnote{We mention that $\opA^+_-\in\call_{\mathrm{iso}}(\overline{H}^{\mu+\nu}(G_-), \overline{H}^{\mu_-+\nu}(G_-))$ with the inverse given by $(\opA^+)^{-1}_-$.}

\medskip

We denote by $\mathcal{N}^{\pm}$ and $\mathcal{N}^{\nu}$ the Newton polygon associated to the order functions $\mu^\pm$ and $\nu$, respectively.
We set $N_\pm:=\ord \mu_\pm$ and $m:=\ord \nu$.
The key for a well-posedness theory for admissible operators on the half plane is the following lemma. 
\begin{lemma}\label{lemma_added_boundary_conditions}
    Let $\opA=\opA^-\opA^+$ be admissible with notations as in Definition \ref{def: admissible} such that $\mu_-+\nu$ is CHG-shaped.
    Additionally, let $\opL\in \call_{\mathrm{iso}}(H^{\mu_-+\nu}_\perp(G),L^2_\perp(G))$ be such that $\opL$ preserves upper and lower support, and $\opL^{-1}$ preserves upper support.
    Suppose that $f\in \overline{H}^{\nu}_\perp(G_+)$ and $u \in \overline{H}^{\mu+\nu}_\perp(G_+)$.
    Then
    \begin{align*}
        \opA_+u=f \quad \Longleftrightarrow \quad
        \left\{ \begin{array}{rcl}
        \opL_+\opA^+_+ u&=& \opL_+  (\opA^-)^{-1}_+f,     \\
        \Tr^{(N_-+m)}\opA^+_+u&=& \Tr^{(N_-+m)}(\opA^-)^{-1}_+f.      
        \end{array}\right.
    \end{align*}
\end{lemma}
\begin{proof} ''$\Rightarrow$'':
    Since $(\opA^-)^{-1}_+$ is the inverse of $\opA^-_+$, we have $\opA^+_+u=(\opA^-)^{-1}_+f\in \overline{H}^{\mu_-+\nu}_\perp(G_+)$, so that the assertion on the right-hand side is fulfilled.

    \medskip
    
    ''$\Leftarrow$'':
    Set $v_1:= \opA^+_+u$, $v_2:=(\opA^-)^{-1}_+f$.
    Then $v_1, v_2\in \overline{H}^{\mu_-+\nu}_\perp(G_+)$ by the mapping properties of $\opA^+_+$ and $(\opA^-)^{-1}_+$.
    Moreover, both $v_1$ and $v_2$ are solutions of 
    \begin{equation*}
        \left\{ \begin{array}{rcl}
        \opL_+v&=& \opL_+(\opA^-)^{-1}_+f,       \\
         \Tr^{(N_-+m)}v&=&\Tr^{(N_-+m)} (\opA^-)^{-1}_+f.   
        \end{array}\right.
    \end{equation*}
    The operator $\opL$ verifies the hypothesis of Proposition \ref{prop: upper_support_preserving} with $\mu+\nu$ replaced by $\mu_-+\nu$.
    Thus, by the uniqueness assertion in Proposition \ref{prop: upper_support_preserving} we deduce $v_1=v_2$.
    This implies $\opA_+u=\opA^-_+v_1=\opA^-_+v_2=f$ and the conclusion follows.
\end{proof}

In order to identify a useful sufficient condition for Fourier multiplier $\op[m]$ to preserve upper or lower support, we recall the following Paley-Wiener type theorem.
\begin{theorem}[Paley-Wiener-Stein]\label{thm_Paley_Wiener}
    Let $f\in \cals(\R)$.
    Then $\supp f\subseteq \overline{\mathbb{R}_-}$ if and only if $\mathcal{F}_\R f$ admits an extension $\widetilde{\mathcal{F}_\R f}\in C_b(\overline{\mathbb{C}_+})$ that is holomorphic on $\mathbb{C}_+$.
    In that case, such an extension is given by 
    \begin{equation*}
        \widetilde{\mathcal{F}_\R f}(z):=\int_{-\infty}^\infty f(x)\mathrm{e}^{-\ic x z}d x,\qquad z\in \overline{\mathbb{C}_+}.
    \end{equation*}
\end{theorem}
\begin{proof}
See \cite[Theorem 3.5]{stein2010complex}.
\end{proof}
\begin{proposition}\label{prop_support_preserving}
    Let $P\in \mathcal{O}(\mathbb{R})$ and
    suppose that $P$ admits an extension $\widetilde{P}\in C(\overline{\mathbb{C}_+})$ that is holomorphic on $\mathbb{C}_+$, and such that there exists a constant $C>0$ and a polynomial $Q$ verifying 
    \begin{equation*}
        \vert \widetilde{P}(z)\vert \leq C\vert Q(z)\vert
    \end{equation*}
    for all $z\in \overline{\mathbb{C}_+}$.
    Then $\op[P]:\cals(\R)\to\mathcal{S}(\mathbb{R})$ preserves lower support. 
\end{proposition}
\begin{proof}
    Let $f\in \dot{\mathcal{S}}(\overline{\mathbb{R}_-})$.
    Since $\calf_\R f\in \cals(\R)$ and $P\in \calo(\R)$, we have $P\calf_\R f \in \cals(\R)$, so that $\op[P]f\in \cals(\R)$.
    Therefore we only need to show that $\supp \op[P]f\subseteq \overline{\R_-}$.

    \medskip
    
    By Theorem \ref{thm_Paley_Wiener}, $\mathcal{F}_\R f$ has an extension $\widetilde{\mathcal{F}_\R f}\in C_b(\overline{\mathbb{C}_+})$ that is holomorphic on $\mathbb{C}_+$.
    Write $Q(z)=\sum_{k=1}^K c_k z^k$.
    Since $\op[Q]=\sum_{k=1}^K c_k (-\ic \partial_x)^k$ is a differential operator, it holds $\op[Q]f\in \dot{\cals}(\overline{\R_-})$.
    Thus, the function $h:=\mathcal{F}_\R\op[Q]f$ admits an extension $\widetilde{h}\in C(\overline{\mathbb{C}_+})$ that is holomorphic on $\mathbb{C}_+$, and given by 
    \begin{align*}
        \widetilde{h}(z)&=\int_{-\infty}^\infty (\op[Q]f)(x)\mathrm{e}^{-\ic x z}\dd x=\sum_{k=1}^K c_k (-\ic)^k\int_{-\infty}^\infty \partial_x^k f(x)\mathrm{e}^{-\ic x z}\dd x\\
        &=\sum_{k=1}^K c_kz^k\int_{-\infty}^\infty f(x)\mathrm{e}^{-\ic x z}\dd x
        =Q(z)\widetilde{\mathcal{F}_\R f}(z).
    \end{align*}
    Consider now $P\mathcal{F}_\R f$.
    By hypothesis, this function admits an extension $\widetilde{P}\widetilde{\mathcal{F}_\R f}$.
    This extension is continuous on $\overline{\mathbb{C}_+}$ and holomorphic on $\mathbb{C}_+$.
    Furthermore, it is bounded on $\overline{\mathbb{C}_+}$ since $\widetilde{h}$ is bounded on $ \overline{\mathbb{C}_+}$ and there it holds
\begin{equation*}
    \vert \widetilde{P}\widetilde{\mathcal{F}_\R f}\vert \leq C\vert Q\widetilde{\mathcal{F}_\R f}\vert =\vert \widetilde{h}\vert.
\end{equation*}
Thus, using once more Theorem \ref{thm_Paley_Wiener}, we deduce $\supp \op[P]f \subseteq \overline{\mathbb{R}_-}$.
\end{proof}
\begin{remark}
    Let $\widehat{m}\in \calo^\perp(\widehat{G})$ be such that for all $k\neq 0$ and $\xi'\in \mathbb{R}^{n-1}$, the operator $\op[m(k, \xi', \cdot)]: \mathcal{S}(\mathbb{R})\to \mathcal{S}(\mathbb{R})$ preserves lower support. Then $\op[m] : \mathcal{S}_\perp(G)\to \mathcal{S}_\perp(G)$ also preserves lower support. Indeed, one has 
    \begin{equation*}
        \op[m]f=\mathcal{F}^{-1}_{G^{n-1}}\op[m(k, \xi', \cdot)]\mathcal{F}_{G^{n-1}},
    \end{equation*}
    and both $\mathcal{F}^{-1}_{G^{n-1}}:\cals(\widehat{G}^{n-1})\to\cals(G^{n-1})$ and $\mathcal{F}_{G^{n-1}}: \mathcal{S}(G^{n-1})\to \mathcal{S}(\widehat{G}^{n-1})$ preserve lower support.
\end{remark}
We next show that we can adapt the  weight functions to define equivalent norms on the Newton polygon spaces on the half space.
\begin{proposition}\label{prop:weight_fct_supp}
    Let $\mu$ be a strictly positive order function with associated Newton polygon $\mathcal{N}$ and its set of vertices $\mathcal{N}_V=\{(r_0, s_0),...,(r_{J+1}, s_{J+1})\}$.
    Assume that $r_j\in \mathbb{N}$ for all $j\in\set{0,..., J+1}$, and that $\mathcal{N}$ is regular in time.
    Then there exist  smooth functions $\omega_\mu^\pm : \mathbb{R}\times \mathbb{R}^n\to \mathbb{C}$ such that $\omega_\mu^\pm $ is $\mathcal{N}$-elliptic for the order function $\mu$, $\op[\omega_\mu^\pm]$ preserves both upper and lower support, and such that $\op[\frac{1}{\omega_\mu^+}]$ preserves upper support and $\op[\frac{1}{\omega_\mu^-}]$ preserves lower support. 
     In particular, for all $u\in \overline{H}^{\mu}_\perp(G_+)$ one has 
     \begin{equation*}
         \|u\|_{\mu, + } \simeq_\mu \|\op[\omega_\mu^-]^+u \|_{L^2(G_+)}.
     \end{equation*}
\end{proposition}
\begin{proof}
    By Proposition \ref{prop: decomposition_elementary_order}, it is enough to show that for any elementary order function $o_y$ with $y\in [0,\infty)$, one can find a suitable weight function $\omega_{o_y}$.
    We choose
    \begin{equation*}
        \omega _{o_y}^\pm(\tau, \xi', \xi_n):=\langle\tau\rangle^y +\langle\xi'\rangle \pm \ic\xi_n.
    \end{equation*}
    Then $\omega_{o_y}$ is smooth.
    For every fixed $(\tau, \xi')\in \mathbb{R}\times  \mathbb{R}^{n-1} $, the operator $\op[\omega_{o_y}^\pm (\tau, \xi', \cdot)]$ is a differential operator and thus preserves both upper and lower support.
    Furthermore, the polynomial $z\mapsto \omega_{o_y}^\pm(\tau, \xi', z):=\langle\tau\rangle^y+\langle\xi'\rangle \pm iz$ has roots only in $\mathbb{C}_\pm$, so that the symbol $\frac{1}{\omega_{o_y}^\pm(\tau, \xi', \cdot)}$ admits a continuous and bounded extension $z\mapsto \frac{1}{\omega_{o_y}(\tau, \xi', z)}$ to $\overline{\mathbb{C}_\mp}$ which is holomorphic in $\C_\pm$, and thus preserves support in $G_\pm$ by Proposition \ref{prop_support_preserving}.

    We now show the $\mathcal{N}$-ellipticity.
    The upper estimate follows by the triangle inequality.
    Moreover,
    \begin{equation*}
        \vert \omega_{o_y}(\tau, \xi', \xi_n)\vert^2= \xi_n^2 +W_{o_y}(\tau, \xi', 0)^2 \gtrsim \xi_n^2 +W_{2o_y}(\tau, \xi',0)^2=W_{2o_y}(\tau, \xi', \xi_n),
    \end{equation*}
    so that the lower estimate follows.
\end{proof}
    We remark that a particular outcome of Proposition \ref{prop:weight_fct_supp} is that we can always use $\op[\omega_{\mu_-+\nu}^+]$ as the operator $\opL$ in Lemma \ref{lemma_added_boundary_conditions}.
\subsection{System at the Boundary for the CHG Determinant}
In this section, we determine a class of boundary operators $\opB:\overline{H}^{\mu_D+\nu}(G_+)\to \prod_{j=1}^M H^{\chi_j}(G^{n-1})$ such that for all $f\in \overline{H}^{\nu}(G_+)$ and $g\in \prod_{j=1}^M H^{\chi_j}(G^{n-1})$ the problem
\begin{align}\label{eqn: main_theorem}
\left\{
\begin{array}{rcl}
     \partial_t u-\Delta(\partial_t-\Delta) u&=&f  \\
     \opB u&=&g\\ 
\end{array}
\right.
\end{align}
is uniquely solvable in $\overline{H}^{\mu_D+\nu}(G_+)$.
In particular, we need to identify $M\in \N$ and strictly positive order functions $\nu$, $\chi_j$ for $j\in\set{1,\ldots,M}$ which ensure this.
We recall that $\partial_t-\Delta(\partial_t-\Delta)=\op[D]$ is the CHG determinant with the symbol $D$ defined in \eqref{eqn: determinant}.
We first show that the CHG determinant is admissible in the sense of Definition \ref{def: admissible}.
The key is the study of its roots along with the Paley-Wiener Theorem \ref{thm_Paley_Wiener}.
\begin{lemma}\label{lemma: roots}
    There exist functions $\rho_j^+: (\mathbb{R}\setminus\{0\})\times  \mathbb{R}^{n-1} \to \mathbb{C}$ and $\rho_j^+: (\mathbb{R}\setminus\{0\})\times  \mathbb{R}^{n-1} \to \mathbb{C}$ $j=1,2$ with the following properties. 
    \begin{enumerate}
        \item\label{lemma: rootsi} For all $(\tau, \xi')\in (\mathbb{R}\setminus\set{0})\times \mathbb{R}^{n-1}$ it holds $D(\tau, \xi', \rho_j^\pm(\tau, \xi'))=0$, i.e., the $\rho_j^\pm(\tau, \xi')$ are the roots of the polynomial $D(\tau, \xi', \cdot)$.
        \item\label{lemma: rootsii} For all $(\tau, \xi')\in (\mathbb{R}\setminus\set{0})\times \mathbb{R}^{n-1}$, the $\rho_j^ \pm(\tau, \xi')$ are pairwise disjoint.
        \item\label{lemma: rootsiii} The functions $\rho_j(\tau, \xi')$ are smooth on $(\mathbb{R}\setminus\{0\})\times  \mathbb{R}^{n-1}$.
        \item\label{lemma: rootsiv} For all $(\tau, \xi')\in (\mathbb{R}\setminus\set{0})\times \mathbb{R}^{n-1}$ it holds 
        \begin{equation*}
            \Im(\rho_j^+(\tau, \xi'))>0, \qquad \Im (\rho_j^-(\tau, \xi')<0, \qquad j=1,2.
        \end{equation*}
        \item\label{lemma: rootsv} The symbol $\rho_1^\pm$ is $\mathcal{N}$-elliptic for the order function $ o_{\frac{1}{2}}(\gamma)=\max\{1, \frac{\gamma}{2}\}$ while $\rho_2^\pm$ is  $\mathcal{N}$-elliptic for the order function $ o_{0}(\gamma)=1$.
        Moreover, $|\arg\rho_2^\pm(\tau,\xi')|\in(\frac\pi 4,\frac{3\pi}{4})$.
    \end{enumerate}
\end{lemma}
\begin{proof}
    For fixed $(\tau, \xi')\in (\mathbb{R}\setminus\set{0})\times \mathbb{R}^{n-1}$, an elementary computation shows that the four roots of $D(\tau, \xi', z)$ are given by 
    \begin{align}\label{def:roots}
        \rho_\epsilon(\tau,\xi'):=\frac{\epsilon_1}{2} \sqrt{ -2\vert \xi'\vert^2-\ic\tau+ \epsilon_2\sqrt{-\tau^2-4\ic\tau}}, \quad \epsilon\in \set{-1,1}^2.
    \end{align}
    where for any complex number $a\in \mathbb{C}\setminus (-\infty, 0)$, $\sqrt{a}$ denotes the principal square root of $a$.
    In particular $\sgn\Re\rho_\epsilon(\tau,\xi')=\epsilon_1$.
    Moreover, one has $\sgn\Im\rho_{\eps}(\tau,\xi')=-\epsilon_1\epsilon_2\sgn(\tau)$.
    Together, this shows that every root is contained in a different quadrant.
    We now set
    \begin{align*}
        \rho_1^-(\tau,\xi')&:=\rho_{(\sgn(\tau),1)}(\tau,\xi'), \quad
        \rho_1^+(\tau,\xi'):=\rho_{(-\sgn(\tau),1)}(\tau,\xi')=-\rho_1^+(\tau,\xi'), \\
        \rho_2^-(\tau,\xi')&:=\rho_{(-\sgn(\tau),-1)}(\tau,\xi'), \quad
        \rho_2^+(\tau,\xi'):=\rho_{(\sgn(\tau),-1)}(\tau,\xi')=-\rho_2^+(\tau,\xi').
    \end{align*}
    Then $\rho^\pm_j$, $\in\set{1,2}$, is smooth on $(\mathbb{R}\setminus\set{0})\times \mathbb{R}^{n-1}$ and conditions \ref{lemma: rootsi}, \ref{lemma: rootsii}, \ref{lemma: rootsiii}, and \ref{lemma: rootsiv} are fulfilled. 
    We now show the point \ref{lemma: rootsv}.
    Due to $\rho_j^+(\tau,\xi')=-\rho_j^-(\tau,\xi')$ we may restrict ourselves to $\rho_j^-$.
    We will repeatedly use that $\rho_2^- (\tau, \xi')^2=\frac{1}{4}(Y(\xi')+Z(\tau))$ with 
    \begin{equation*}
        Y(\xi'): =-2\vert \xi'\vert, \qquad Z(\tau):=-\ic\tau-\sqrt{-\tau^2-4\ic\tau}.
    \end{equation*}
    We recall the formula $\sqrt{z}=\sqrt{\frac{|z|+x}{2}}+i\sgn(y)\sqrt{\frac{|z|-x}{2}}$ for the principle square root of a complex number $z=x+\ic y$.
    Applied to $x=-\tau^2$ and $y=-4\tau$, this yields due to $\sgn(y)=-\sgn(\tau)$
    \begin{align*}
        \Re Z(\tau)&= -\sqrt{\frac{\sqrt{\tau^4+16\tau^2}-\tau^2}{2}}<0,
        \quad
        \Im Z(\tau)= -\tau+\sgn\tau\sqrt{\frac{\sqrt{\tau^4 +16\tau^2}+\tau^2}{2}}.
    \end{align*}
    We are now in the position to prove point \ref{lemma: rootsv} in three short steps.
    ~\begin{itemize}[leftmargin=*]
        \item We first observe that due to $\Re(\rho_2^-(\tau,\xi')^2)<0$ we obtain $|\arg\rho_2^-(\tau,\xi')|\in (\frac\pi 4,\frac{3\pi}{4})$.
        \item We next show that $\vert \rho^-_1\vert\lesssim W_{o_\frac{1}{2}}$ and  $\vert \rho^-_2\vert\lesssim W_{o_0}$
        For $\rho^-_1$, this entails $|\rho^-_1(\tau,\xi')|\lesssim 1+|\xi'|+|\tau|^\frac12$, which follows readily from \eqref{def:roots}.
        For $\rho^-_2$, we need to show $|\rho^-_2(\tau,\xi')|\lesssim 1+|\xi'|$.
        Observe that $|\Re Z(\tau)|$ is increasing in $|\tau|$ and
        \begin{equation*}
            \lim\limits_{|\tau| \to \infty} |\Re Z(\tau)| =\lim\limits_{\tau \to \infty}\sqrt{\frac{\sqrt{\tau^4+16\tau^2}-\tau^2}{2}}=\lim_{\tau\to\infty}\sqrt{\frac{8}{\sqrt{1+\frac{16}{\tau^2}}+1}}=2.
        \end{equation*}
        Moreover $\Im\rho_2^-(\tau,\xi')^2=\Im Z(\tau)$ is continuous and  
        \begin{equation*}
            \lim_{|\tau|\to\infty} |\Im Z(\tau)|=\lim_{|\tau|\to\infty}-|\tau|+\sqrt{\frac{\sqrt{\tau^4 +16\tau^2}+\tau^2}{2}}=\lim_{|\tau|\to\infty}\frac{\sqrt{2}}{|\tau|}\frac{|\Re Z(\tau)|}{\sqrt{2}+\sqrt{\sqrt{1 +\frac{16}{\tau^2}}+1}}=0.
        \end{equation*}
        Thus we obtain $\vert \rho_2^-(\tau,\xi')\vert^2=\vert \rho_2^-(\tau,\xi')^2\vert \leq |Y(\xi')| + |\Re(Z(\tau)| + |\Im(Z(\tau)| \lesssim 1+|\xi'|^2$ as claimed.
        \item We now fix $\lambda>0$ and show that $\vert \rho_1^-(\tau, \xi')\vert \gtrsim_\lambda W_{o_{\frac{1}{2}}}(\tau, \xi')$ when $\vert \tau\vert >\lambda$.
        To this effect, we recall that the CHG determinant $D$ is $\mathcal{N}$-elliptic (Proposition \ref{prop: det_elliptic}), that is,
        \begin{align*}
            \vert D(\tau, \xi', \xi_n)\vert \gtrsim_\lambda W_{\mu_D}(k, \xi', \xi_n)   
        \end{align*}
        for all $\vert \tau\vert >\lambda$.
        Furthermore, one has $\rho_1^+(k, \xi')\rho_1^-(k, \xi')\rho_2^+(\tau, \xi')\rho_2^-(\tau, \xi')=D(\tau, \xi', 0)$ for all $\tau\neq 0$ and $\mu_D= 2o_\frac12+2o_0$. Using the previous point and Lemma \ref{lemma: additivity}, one then gets 
        \begin{equation*}
            \vert \rho_1^+(\tau, \xi')\vert
            =\frac{\vert D(\tau, \xi', 0)\vert}{\vert\rho_1^-(\tau, \xi')\rho_2^+(\tau, \xi')\rho_2^-(\tau, \xi')\vert }
            \gtrsim_\lambda \frac{W_{\mu_D}(\tau, \xi', 0)}{CW_{o_\frac{1}{2}}(\tau, \xi')W_{o_0}(\tau, \xi')W_{o_0}(\tau, \xi')} 
            \gtrsim_\lambda W_{o_\frac{1}{2}}(\tau, \xi')
        \end{equation*}
        for all $\vert \tau\vert>\lambda$.
        By analogy, the estimates from below for the other roots follow. 
    \end{itemize}
\end{proof}
\begin{proposition}\label{prop: DHG_admissible}
     There exists symbols $D^+, D^-: (\R\setminus\set{0})\times \mathbb{R}^n\to \mathbb{C}$ such that $\op[D]=\op[D^+]\op[D^-]$ is admissible in the sense of Definition \ref{def: admissible} for the CHG order function $\mu_D$ and any strictly positive order function $\nu$.
     Moreover, $D^\pm$ is $\mathcal{N}$-elliptic with the order function $\mu_\pm:=\frac12\mu_D$. 
\end{proposition}
\begin{proof}
    Denote by $\rho_j^\pm(\tau, \xi')$, $j\in\set{1,2}$, the roots of $D(\tau, \xi', \cdot)$ from Lemma \ref{lemma: roots} and define 
    \begin{equation*}
        D^-(\tau, \xi', \xi_n):=(\xi_n-\rho_1^-(\tau, \xi'))(\xi_n-\rho_2^-(\tau, \xi')), \qquad D^+(\tau, \xi', \xi_n):=(\xi_n-\rho_1^+(\tau, \xi'))(\xi_n-\rho_2^+(\tau, \xi')).
    \end{equation*}
    By Lemma \ref{lemma: roots}, the roots $\rho_1^\pm$ (resp. $\rho_2^\pm)$ have an upper order function $o_{\frac{1}{2}}$ (resp. $o_0$). It follows easily that every of the factors $\xi_n-\rho_1^\pm$ (resp. $\xi_n-  \rho_2^\pm$ ) also has upper order function $o_{\frac{1}{2}}$ (resp. $o_0$). Since furthermore 
    \begin{equation*}
        D(\tau, \xi', \xi_n)=(\xi_n-\rho_1^-(\tau, \xi'))(\xi_n-\rho_2^-(\tau, \xi'))(\xi_n-\rho_1^+(\tau, \xi'))(\xi_n-\rho_2^+(\tau, \xi'))
    \end{equation*}
    and we know that $D$ is $\mathcal{N}$-elliptic for the order function $ \mu_D=2o_{\frac{1}{2}}+2o_0$, an argument similar to the one in the proof of Lemma \ref{lemma: roots} shows that both $D^-, D^+$ are $\mathcal{N}$-elliptic with order function $\mu_{\pm}=\frac12\mu_D$. 

    \medskip
    
    From Theorem \ref{thm: whole space} we learn $\op[D^-]\in \call_{\mathrm{iso}}(H^{\mu_-+\nu}_\bot(G),H^{\nu}_\bot(G))$ and $\op[D^+]\in \call_{\mathrm{iso}}(H^{\mu_D+\nu}_\bot(G),H^{\mu_-+\nu}_\bot(G))$ with respective inverses $\op[D^\pm]^{-1}=\op[\frac1{D^\pm}]$, and it holds $\op[D]=\op[D^-]\op[D^+]$.
    Moreover, $\frac{1}{D^-}(\tau,\xi',\cdot)$ has an extension $\frac{1}{\widetilde{D^-}}(\tau,\xi',\cdot)\in C_b(\overline{\C_+})$ that is holomorphic on $\mathbb{C}_+$, thus verifying the hypothesis of Proposition \ref{prop_support_preserving}.
    Indeed, since $\Im(\rho_1^-)<0$ and $\Im(\rho_2^-)<0$ one can take 
    \begin{equation*}
        \frac{1}{\widetilde{D^-}(\tau, \xi', z)}:=\frac{1}{(z-\rho_1^-(\tau, \xi'))(z-\rho_2^-(\tau, \xi'))}, \quad z\in \overline{\mathbb{C}_+}. 
    \end{equation*}
    Thus $\op[D^-]^{-1}=\op[\frac1{D^-}]$ preserves lower support.
    By analogy, $\op[D^+]^{-1}$ preserves upper support.
\end{proof}
We denote by $d^+_j$, $j\in\set{0,1,2}$, the coefficients of $D^+$, i.e.
\begin{equation*}
    D^+(\tau, \xi', \xi_n)=:d^+_2(\tau, \xi')(\ic\xi_n)^2+ d^+_1(\tau, \xi')(\ic\xi_n) + d_0^+(\tau, \xi'). 
\end{equation*}
Remark that $d^+_2 = -1$, $d^+_1=\ic(\rho_1^++\rho_2^+)$, and $d^+_0=\rho_1^+\rho_2^+$.
In view of Proposition \ref{prop: upper_support_preserving}, it seems reasonable to complement the part $\op[D^+]$ of the CHG determinant by two boundary operators.
We will allow those to be of the form
\begin{equation}\label{eqn: bound_op_LS}
    \opB_i=\sum_{j=0}^3 \op[b^{(i)}_j]\Tr_{j}, 
\end{equation}
where the $b^{(i)}_j$ are smooth symbols on $(\mathbb{R}\setminus\set{0})\times \mathbb{R}^{n-1}$.
Observe that additionally
\begin{equation*}
    \Tr_i \op[D^+]_+=\sum_{j=0}^{2} \op[d_j^+] \Tr_{i+j}, \qquad i\in \mathbb{N}.
\end{equation*}
This motivates the definition of the following boundary matrix, which will turn out to be the cornerstone of our theory.

\begin{definition}\label{def:compl_cond}
    Let $D=D^+D^-$ be the symbol of the CHG determinant with order function $\mu_D=\mu_++\mu_-$ and let $\opB:=(\opB_1,\opB_2)$ be as in \eqref{eqn: bound_op_LS}.
    Moreover, suppose that $m\in \mathbb{N}_0$.
    \begin{enumerate}
        \item\label{def:compl_condi} The \emph{extended boundary matrix} associated to $(D,\opB)$ is the matrix symbol $\mathcal{B}:=\mathcal{B}^{(m)}: (\R\setminus\set{0})\times \R^{n-1}\to \mathbb{C}^{(4+m)\times(4+m)}$ defined by 
        \begin{align*}
            \mathcal{B}^{(m)}_{ij}
            :=\begin{cases}
            b^{(i)}_{j-1}     & \text{if } i\in\set{1, 2},\\
            d^+_{j+2-i} & \text{if } i\in{3,..., 4+m},
            \end{cases}
        \end{align*}
        where $j\in\set{1,..., 4+m}$ and we understand $d^+_{k}=0$ if $k<0$ or $k>2$.
        \item\label{def:compl_condii} We say that $\opB$ fulfills the complementing boundary conditions for $D$ with respect to the order functions $\chi_1$, $\chi_2$, and $\nu$, if the extended matrix $\mathcal{B}^{(\ord\nu)}$ is a mixed-order system in the sense of Definition \ref{def: mixed_order_systems} with
        \begin{align*}
            t_i&:=T_{i-1}(\mu+\nu), \qquad i\in\set{1,\ldots,4+m},\\
            s_i&:=
            \begin{cases}
                -\chi_i, & \text{if } i\in\set{1,2},\\
                -T_{i-3}(\mu_-+\nu), & \text{if } i\in\set{3,\ldots,4+m}
            \end{cases}
        \end{align*}
    \end{enumerate}
\end{definition}
Observe that the complementing boundary conditions imply that $\det\calb(\tau,\xi')\ne 0$ for all $(\tau,\xi')\in(\R\setminus\set{0})\times\R^{n-1}$, but that they go beyond this criterion.
Indeed, the idea of Definition \ref{def:compl_cond} is to study the invertibility of the matrix symbol $\op[\mathcal{B}]$ in the trace spaces, treating it as a mixed-order system.
Accordingly, the order functions $t_i$ associated to the columns are those of the trace spaces of $H^{\mu}_\perp(G)$.
For the order functions associated to the rows, the first two are those of the target spaces of the boundary operators $\opB_i$, while the remaining order functions are those of the target spaces of $\Tr_i\op [D^+]_+$.
\begin{theorem}\label{thm: half_spcae_general}
    Let $\opB:=(\opB_1,\opB_2)$ fulfill the complementing boundary conditions for $D$ with respect to the strictly positive order functions $\chi_1$, $\chi_2$, and $\nu$, where $\mu_D+\nu$ is CHG-shaped.
    Then for all $f\in \overline{H}^{\nu}_\perp(G_+)$ and $g=(g_1, g_2)$ with $g_i\in H^{\chi_i}_\perp(G^{n-1}) $, $i\in\set{1,2}$, the problem \eqref{eqn: main_theorem} admits a unique solution $u\in \overline{H}^{\mu_D+\nu}_\perp(G_+)$, and it holds
    \begin{equation*}
        \|u\|_{\mu_D+\nu, +}\lesssim_{\nu,\chi_1,\chi_2,\opB} \|f\|_{\nu, +}+\|g_1\|_{\chi_1}+\|g_2\|_{\chi_2}.
    \end{equation*}
\end{theorem}
\begin{proof}
    With $\mu_D+\nu$ also $\mu_-+\nu$ is CHG-shaped. Thus, by Lemma \ref{lemma_added_boundary_conditions}, problem (\ref{eqn: main_theorem}) is equivalent to 
    \begin{align*}
        \left\{ \begin{array}{rcll}
        \op[\omega_{\mu_-+\nu}^+D^+]_+u&=&\op[\omega_{\mu_-+\nu}^+]_+\op[D^-]^{-1}_+f    &   \\
        \Tr_i\op[D_+]_+u&=&\Tr_i\op[D^-]^{-1}_+f, & \quad i\in\set{0,\ldots, 1+m},\\
        \opB_i u&=&g_i , & \quad i\in\set{1,2}.
        \end{array}\right.
    \end{align*}
    Since $\mathcal{B}$ is a mixed-order system, setting $h=(h_1,..., h_{4+m})$ with $h_1:=g_1, h_2:=g_2, h_{3+i}:=\Tr_{i}(\op[D^-]^{-1}_+f)$ with $i\in\set{0,..., 1+m}$, we have $h_i\in H^{-s_i}_\perp(G^{n-1})$ for $i\in\set{1,..., 4+m}$.
    Thus, by Theorem \ref{thm: whole_space_systems}, there exists a unique $l=(l_1,..., l_{4+m})$ with $l_i\in H^{t_i}_\perp(G^{n-1})$ such that $\op[\mathcal{B}] l=h$, and
    \begin{equation*}
        \sum_{i=1}^{4+m} \|l_i\|_{t_i}\lesssim \sum_{i=1}^{4+m} \|h_i\|_{-s_i}. 
    \end{equation*}
    By Proposition \ref{prop: upper_support_preserving}, the problem 
    \begin{align*}
        \left\{\begin{array}{rcl}
         \op[\omega_{\mu_-+\nu}^+D^+]_+u&=&\op[\omega_{\mu_-+\nu}^+]_+\op[D^-]^{-1}_+f,       \\
         \Tr^{(4+m)}u&=&l,   
        \end{array}\right. 
    \end{align*}
    admits a unique solution $u\in \overline{H}^{\mu+\nu}_\perp(G_+)$.
    Observe that $u$ is a solution to \eqref{eqn: main_theorem}.
    Furthermore, again by Proposition \ref{prop: upper_support_preserving}, we have
    \begin{equation*}
        \|u\|_{\mu, +}\lesssim \|\op[\omega_{\mu_-+\nu}^+]_+\op[D^-]^{-1}_+f\|_{L^2(G_+)} + \sum_{i=1}^{4+m} \|l_i\|_{t_i}. 
    \end{equation*}
    Since both $\op[\omega_{\mu_-}^+]$ and $\op[D^-]^{-1}$ preserve lower support, we obtain by Proposition \ref{prop: boundedness_lower_support}
    \begin{equation*}
        \|\op[\omega_{\mu_-+\nu}^+]_+\op[D^-]^{-1}_+f\|_{L^2(G_+)}\lesssim \|\op[D^-]^{-1}_+f\|_{\mu_-+\nu, +}\lesssim \|f\|_{\nu, +}.
    \end{equation*}
    On the other hand, the boundedness of the trace operator shown in Theorem \ref{thm: trace} combined with the above gives 
    \begin{align*}
        \sum_{i=1}^{4+m} \|l_i\|_{t_i} &\lesssim  \sum_{i=1}^{4+m} \|h_i\|_{-s_i}
        =\|g_1\|_{\chi_1}+\|g_2\|_{\chi_2}+\sum_{j=1}^{2+m} \|\Tr_{j-1}(\op[D^-]^{-1}f)\|_{T_j(\mu_-+\nu)}\\
        &\lesssim  \|g_1\|_{\chi_1}+\|g_2\|_{\chi_2}+\|\op[D^-]^{-1}_+f\|_{\mu_-+\nu, +}
        \lesssim \|g_1\|_{\chi_1}+\|g_2\|_{\chi_2}+\|f\|_{\nu, +}.
        \qedhere
    \end{align*}
\end{proof}
\begin{corollary}\label{cor: corollary_trace_1_3}
    Let $D$ be the CHG determinant in \eqref{eqn: determinant},  $\mu_D$ the associated order function and consider the problem 
    \begin{equation}\label{eqn: corollary_trace_1_3}
        \left\{\begin{array}{rcl}
         \partial_t u-\Delta(\partial_t-\Delta) u&=&f,      \\
        \Tr_1 u&=&g_1,\\
        \Tr_3 u&=&g_2.
        \end{array}\right.
    \end{equation}
    \begin{enumerate}
        \item\label{cor: corollary_trace_1_3i} For all $f\in L^2_\perp(G_+)$, $g_1\in H^{T_1(\mu_D)}_\perp(G^{n-1})$, and $g_2\in H^{T_3(\mu_D)}_\perp(G^{n-1})$, problem \eqref{eqn: corollary_trace_1_3} admits a unique solution $u\in \overline{H}^{\mu_D}_\perp(G_+)$, and it holds
        \begin{equation*}
            \|u\|_{\mu_D, +}\lesssim \|f\|_{L^2(G_+)}+\|g_1\|_{T_1(\mu_D)}+ \|g_2\|_{T_3(\mu_D)}.
        \end{equation*}
        \item\label{cor: corollary_trace_1_3ii} Consider the order function $\nu\equiv 1$.
        For all $f\in \overline{H}^{\nu}_\perp(G_+)$, $g_1\in H^{T_1(\mu_D+\nu)}_\perp(G^{n-1})$, and $g_2\in H^{T_3(\mu_D+\nu)}_\perp(G^{n-1})$, problem \eqref{eqn: corollary_trace_1_3} admits a unique solution $u\in \overline{H}^{\mu_D+\nu}_\perp(G_+)$, and it holds
        \begin{equation*}
            \|u\|_{\mu_D+\nu, +}\lesssim \|f\|_{\nu, +}+\|g_1\|_{T_1(\mu_D+\nu)}+ \|g_2\|_{T_3(\mu_D+\nu)}.
        \end{equation*}
    \end{enumerate}
\end{corollary}
\begin{proof}
\begin{enumerate}
    \item  With the notations of Theorem \ref{thm: half_spcae_general}, the associated extended boundary matrix is 
    \begin{equation*}
        \mathcal{B}=\begin{pmatrix}
            0 &1&0&0\\
            0&0&0&1\\
            d^+_0& d^+_1 & -1 & 0\\
            0&d^+_0 & d^+_1 & -1
        \end{pmatrix}. 
    \end{equation*}
    We take $\mu:=\mu_D$ and $\chi_1:=T_1(\mu_D), \chi_2:=T_3(\mu_D)$.
    In view of the computations made in Example \ref{ex: trace_CHG}, the order functions associated to the columns are $t_1=2o_\frac{1}{2}+\frac32$, $t_2=2o_{\frac{1}{2}} +\frac12$, $t_3=\frac32 o_{\frac{1}{2}}$, $t_4= \frac12 o_{\frac{1}{2}}$. Similarly, the order functions associated to the rows are $s_1=-\chi_1=-2o_\frac12-\frac12$, $s_2=-\chi_2=-\frac12 o_{\frac{1}{2}}$,  $s_3=-T_0(\mu_+)=-o_{\frac{1}{2}}-\frac12$, $s_4=-T_1(\mu_+)=-\frac12 o_{\frac{1}{2}}$.
    Together, we get 
    \begin{equation*}
        \delta= \sum_{k=1}^4 s_k+t_k = 2o_{\frac{1}{2}}+1
    \end{equation*}
    and $\det \mathcal{B}=-d^+_0d^+_1$.
    By Lemma \ref{lemma: ellipticity_coeff} below, $d^+_0d^+_1$ is $\mathcal{N}$-elliptic with order function $2o_{\frac{1}{2}}+1$.
    Thus Theorem \ref{thm: half_spcae_general} gives the claim. 
   \item Since $\ord(\nu)=1$, the associated extended boundary matrix is
   \begin{equation*}
       \mathcal{B}^{(1)}=\begin{pmatrix}
           0 &1& 0& 0&0\\
           0&0&0&1&0\\
           d_0^+& d_1^+& -1 &0&0\\
           0&d_0^+ & d_1^+&-1 &0\\
           0&0&d_0^+&d_1^+&-1 
       \end{pmatrix}.
   \end{equation*}
   Using Theorem \ref{thm: trace}, the order functions associated to the columns are $t_1=2o_{\frac{1}{2}}+\frac52$, $t_2=2o_{\frac{1}{2}}+\frac32$, $t_3= 2o_{\frac{1}{2}}+\frac12$, $t_4=\frac32 o_{\frac{1}{2}}$ and $t_5=\frac12 o_{\frac{1}{2}}$.
   The order functions associated to the rows are $s_1= -2o_{\frac{1}{2}}-\frac32$, $s_2=-\frac32 o_{\frac{1}{2}}$, $s_3=-o_{\frac{1}{2}}-\frac32$, $s_4=-o_{\frac{1}{2}}-\frac12$ and $s_5=-\frac12o_{\frac{1}{2}}$.
   Once again, we have 
   \begin{equation*}
       \delta= \sum_{k=1}^5 s_k+t_k=2o_{\frac{1}{2}}+1
   \end{equation*}
   and $\det \mathcal{B}^{(1)}=d^+_0d^+_1$.
   The conclusion follows as above with Lemma \ref{lemma: ellipticity_coeff} below.
\end{enumerate}
   
\end{proof}
\begin{lemma}\label{lemma: ellipticity_coeff}
    The symbol $d^+_1$ is $\mathcal{N}$-elliptic for the order function $o_{\frac{1}{2}}$, and the symbol $d^+_0$ is $\mathcal{N}$-elliptic for the order function $\mu_+=o_{\frac12}+1$.
    In particular, $d^+_0d^+_1$ is $\mathcal{N}$-elliptic for the order function $2o_\frac12+1$.
\end{lemma}
\begin{proof}
    Since $d^+_0(\tau,\xi')=D^+(\tau,\xi',0)$, the $\mathcal{N}$-ellipticity of $d^+_0$ follows directly from the one of $D^+$ in Proposition \ref{prop: DHG_admissible}.
    It hence suffices to show that $d^+_0$ is $\mathcal{N}$-elliptic for the order function $\mu_+=o_{\frac12}+1$.
    
    \medskip
    
    Recall that $d^+_1=\rho_1^++  \rho_2^+$, where $\rho_1^+, \rho_2^+$ are the roots chosen in Lemma \ref{lemma: roots}.
    The upper estimates for $d^+_1$ follow via the triangle inequality from the $\mathcal{N}$-ellipticity of said roots.
    For the lower estimates, we observe that for any two complex numbers $z, w$ with $\vert \arg z-\arg w\vert< 2\omega <\pi$, it holds $\vert z+w\vert \geq \cos(\omega)(\vert z\vert+\vert w\vert)$, see e.g. \cite[Lemma C.2]{bravin2025well}.
    Since $|\arg\rho_2^+|\in (\frac\pi 4,\frac{3\pi}{4})$ by Lemma \ref{lemma: roots}\ref{lemma: rootsv}, we have $|\arg \rho_1^+-\arg \rho_2^+|<2\omega$ with $\omega:=\frac{3\pi}{8}$.
    Thus, for any $\lambda>0$, we have
    \begin{equation*}
        \vert\rho_1^++\rho_2^+\vert
        \gtrsim \vert\rho_1^+\vert+ \vert \rho_2^+\vert
        \gtrsim_{\lambda} W_{o_{\frac{1}{2}}}+W_{o_0}
        \gtrsim W_{o_{\frac{1}{2}}},
    \end{equation*}
    where the second inequality follows from the $\mathcal{N}$-ellipticity of $\rho_1^+$ and $\rho_2^+$ in Lemma \ref{lemma: roots}.
\end{proof}

\section{Proofs of the Main Theorems}\label{sec:proof}
    The careful study of the determinant $D$ of the Cahn-Hilliard-Gurtin system \eqref{eqn: CHG_system_1} on the half space which we carried out in Section \ref{sec:half} allows us to solve the system efficiently. 

\begin{proof}[Proof of Theorem \ref{thm: main}]
    We denote by $L(\tau, \xi)$ the symbol matrix in \eqref{eqn: CHG_symbol} associated to the system \eqref{eqn: CHG_system_1}.
    The existence of a bounded right-inverse for the trace operator proved in Theorem \ref{thm: trace} reduces the problem to the case $g_1=g_2=0$. 
    We set $\mu_D:=2+2o_{\frac12}$ and
    \begin{equation*}
        \mathbb{V}:= \overline{H}^{\mu_D}_\perp(G_+) \times \overline{H}^{\mu_D+1}_\perp(G_+)
    \end{equation*}
    Using both parts of Corollary \ref{cor: corollary_trace_1_3} and $\op[D]_+=\partial_t-\Delta(\partial_t-\Delta)$, one can find $v=(v_1,v_2)\in\mathbb{V} $ such that 
    \begin{equation*}
        \left\{\begin{array}{rcl}
             \op[D]_+v_i&=&f_i,\\
             \Tr_1 v_i&=&0,\\
             \Tr_3 v_i&=&0,
        \end{array}\right. \qquad i\in\set{1,2}.
    \end{equation*}
    Consider now the symbol matrix 
    \begin{equation*}
        \ad L(\tau, \xi)=\begin{pmatrix}
            1& -\vert\xi\vert^2\\\ic\tau+\vert \xi\vert^2& \ic\tau.
        \end{pmatrix}
    \end{equation*}
     Combining Propositions \ref{prop: ad}, \ref{prop: boundedness_lower_support} and Example \ref{ex: CHG_mixed_order_system}, we get that $\op[\ad L]_+: \mathbb{V}\to \mathbb{E}$ is well-defined and bounded. Thus we can set  $u:=\op[\ad L]_+v\in \mathbb{E}$ and obtain that $(\op[L]_+u)_i=\op[D]_+v_i=f_i$ for $i\in\set{1,2}$. 
    Furthermore, 
    \begin{equation*}
        \Tr_1 u_1=\Tr_1(v_1 -\op[\vert \xi\vert^2]_+v_2)=\Tr_1 v_1-\op[\vert \xi'\vert^2]\Tr_1 v_2-\Tr_3 v_2=0
    \end{equation*}
    as well as 
    \begin{equation*}
        \Tr_1 u_2=\Tr_1(\op[i\tau]_+v_1 +\op[i\tau+\vert \xi\vert^2]_+v_2)=\op[i\tau]\Tr_1 v_1+\op[i\tau+\vert \xi' \vert^2]\Tr_1 v_2+\Tr_3 v_2=0. 
    \end{equation*}
    Thus $u$ is a solution to \eqref{eqn: CHG_system_1}. The estimates follow from those of Corollary  \ref{cor: corollary_trace_1_3} and from the boundedness of $\op[\ad L]_+:\mathbb{V}\to\mathbb{E}$.
    
    \medskip
    
    We now show uniqueness. Assume that $u=(u_1, u_2)$ solves 
    \begin{equation*}
        \left \{ \begin{array}{rcl}
         \op[L]_+u&=&0,      \\
         \Tr_1 u_1&=&0,    \\
         \Tr_1 u_2&=&0.
        \end{array}\right. 
    \end{equation*}
    Then multiplying both sides of the equation $\op[L]_+u=0$ by $\op[A]_+$ gives in particular $\op[D]_+u_1=0$ and so $\op[D^+]_+u_1=0$, since $\op[D^-]_+\in \call_{\mathrm{iso}}(\overline{H}^{\mu_-}_\perp(G_+),L^2_\perp(G_+))$ for $\mu_-=1+o_\frac12$ by Theorem \ref{thm: whole space}, Proposition \ref{prop: boundedness_lower_support}, and Proposition \ref{prop: DHG_admissible}. 

    Remark that $\xi_n^2= -d_1^+(\tau, \xi')(i\xi_n)-d_0^+(\tau, \xi') + D^+(\tau, \xi', \xi_n)$, and in particular
    \begin{align}\label{eqn: trace_unique}
        -\op[\xi_n^2]_+u_1=\op[d_1^+(\tau, \xi')(\ic\xi_n)]_+u_1+\op[d_0^+(\tau, \xi')]_+ u_1    
    \end{align}
    due to $\op[D^+]_+u_1=0$.
    We will use \eqref{eqn: trace_unique} in combination with $\Tr_j\op[b(\tau,\xi')(\ic\xi_n)]_+=\op[b(\tau,\xi')]\Tr_{j+1}$ twice:
    We first take $\Tr_0$ on both sides and learn in light of $\Tr_1 u_1=0$ that $\Tr_2 u_1=\op[d_0^+(\tau, \xi')]\Tr_0 u_1$.
    Secondly we apply $\Tr_1$ and obtain
    \begin{align*}
    \Tr_3 u_1=\op[d_1^+(\tau,\xi')]\Tr_2u_1=\op[d_1^+(\tau, \xi')d_0^+(\tau, \xi')]\Tr_0u_1.
    \end{align*}
    But by assumption it holds $\op[-|\xi'|^2-\xi_n^2-\ic\tau]_+ u_1 + u_2=0$, which after an application of $\Tr_1$ gives $ \Tr_3 u_1=0$ due to $\Tr_1 u_1=\Tr_1 u_2=0$.
    This shows $\op[d_1^+(\tau, \xi')d_0^+(\tau, \xi')]\Tr_0u_1=0$, which is updated to $\Tr_0 u_1=0$ in view of the $\mathcal{N}$-ellipticity of $d_1^+$ and $d_0^+$ in Lemma \ref{lemma: ellipticity_coeff}.
    In particular, $u_1$ is a solution in $ \overline{H}^{\mu_D}_\perp(G_+)\subseteq \overline{H}^{\mu_+}_\perp(G_+)$ of 
    \begin{equation*}
        \left\{\begin{array}{rcl}
        \op[D^+]_+u_1&=&0,      \\
        \Tr_0 u_1&=&0,\\
        \Tr_1 u_1&=&0, 
        \end{array}\right.
    \end{equation*}
    where we recall that $\ord\mu_+=2$ for the order function associated to $D^+$.
    By Proposition \ref{prop: upper_support_preserving}, this implies $u_1=0$.
    Plugging $u_1=0$ into \eqref{eqn: CHG_system_1}, we also deduce $u_2=0$, which achieves the proof.
\end{proof}
\begin{proof}[Proof of Theorem \ref{thm: main_2}]
    By Proposition \ref{prop_potential_spaces} and Theorem \ref{thm: trace} it holds $\mathbb{E}=\overline{H}^{\mu_D}_{\perp}(G_+)$ and $\mathbb{G}=H^{T_2(\mu_D)}_\perp(G^{n-1})$.
    Assume first that $u\in \overline{H}^{\mu_D}_{\perp}(G_+)$ is a solution to \eqref{eqn: CHG_Dirichlet}.
    By Lemma \ref{lemma_added_boundary_conditions} and $\Tr_0 u=\Tr_1 u=0$, we conclude
    \begin{align*}
        \Tr_0 \op[D^-]^{-1}_+ f&=\Tr_0 \op[D^+]_+u
        =\op[d^+_0(k,\xi')]\Tr_0 u+\op[d^+_1(k, \xi')]\Tr_1 u+ \Tr_2 u
        =\Tr_2 u\in  H^{T_2(\mu_D)}_\perp(G^{n-1}).
    \end{align*}
    Conversely, assume $\Tr_0(\op[D^-]^{-1}f)_+\in H^{T_2(\mu_D)}_\perp(G^{n-1})$.
    By Lemma \ref{lemma_added_boundary_conditions}, solving problem \eqref{eqn: CHG_Dirichlet} is equivalent to solving 
    \begin{equation}\label{eqn: intermediate_problem_Dirichlet}
         \left\{\begin{array}{rcl}
             \op[\omega_{\mu_-}^+ D^+]_+u&=&\op[\omega_{\mu_-}^+]_+\op[D^-]^{-1}_+f,  \\
             \Tr_0 \op[D^+]_+u&=& \Tr_0 \op[D^-]^{-1}_+ f, \\
             \Tr_1 \op[D^+]_+u&=& \Tr_1 \op[D^-]^{-1}_+f, \\
             \Tr_0 u&=&0,\\\Tr_1 u&=&0.     
        \end{array}\right.
    \end{equation}
    We define $g=(g_1, g_2 ,g_3, g_4) $ via $g_1=g_2=0$ and
    \begin{align*}
        &g_3=\Tr_0 \op[D^-]^{-1}_+f,\\
        &g_4= -\Tr_1 \op[D^-]^{-1}_+f +\op[d_1^+]\Tr_0 \op[D^-]^{-1}_+f.
    \end{align*}
    The hypothesis $\Tr_0 \op[D^-]^{-1}_+f\in H^{T_2(\mu_D)}_\perp(G^{n-1})$ guarantees that $g_i\in H^{T_{i-1}(\mu_D)}_\perp(G^{n-1})$ for all $i\in\set{1,.., 4}$.
    Indeed, for $i\in\set{1,2}$ this is trivial, for $i=3$ this is the exact statement of the hypothesis, and for $i=4$ this follows since $T_3(\mu_D)=T_1(\mu_{D^-})$, see Figure \ref{fig: trace_spc_2}, and since $\op[d_1^+]\in \call(H^{T_2(\mu_{D})}(G^{n-1}),H^{T_3(\mu_{D})}(G^{n-1}))$ as a result of Theorem \ref{thm: whole space}, Lemma \ref{lemma: ellipticity_coeff}, and $T_2(\mu_{D})-T_3(\mu_{D})=o_{\frac12}$.
    By Proposition \ref{prop: upper_support_preserving}, one can find $u\in \overline{H}^{\mu_D}_\perp(G_+)$ such that 
    \begin{equation*}
        \left\{\begin{array}{rcl}
            \op[\omega_{\mu_-}^+D^+]_+u&=&\op[\omega_{\mu_-}^+]_+\op[D^-]_+^{-1}f  \\
             \Tr^{(4)} u&=&g. 
        \end{array}\right.
    \end{equation*}
    Then $u$ is also solution of (\ref{eqn: intermediate_problem_Dirichlet}) and thus of \eqref{eqn: CHG_Dirichlet}.

    \medskip

    Finally, we show that there exists $f\in L^2_\perp(G_+)$ such that $\Tr_0\op[D^-]_+^{-1}f\notin H^{T_2(\mu_D)}_\perp(G^{n-1})$.
    Namely, choose $g\in H^{T_0(\mu_{D^-})}_\perp(G^{n-1})\setminus H^{T_2(\mu_D)}_\perp(G^{n-1})$.
    The latter space is nonempty, see Figure \ref{fig: trace_spc_2}.
    By Theorem \ref{thm: trace} there exists $F\in \overline{H}^{\mu_{D^-}}_\perp(G_+)$ with $\Tr_0 F=g$.
    Since $\op[D^-]^+\in \call_{\mathrm{iso}}(\overline{H}^{\mu_{D^-}}_\perp(G_+),L^2_\perp(G_+))$ by Proposition \ref{prop: boundedness_lower_support} and Proposition \ref{prop: DHG_admissible}, the choice $f:=\op[D^-]^+F\in L^2_\perp(G_+)$ gives $\Tr_0\op[D^-]_+^{-1}f=\Tr_0 F=g\notin H^{T_2(\mu_D)}_\perp(G^{n-1})$.
\begin{figure}[h]
   \begin{subfigure}[b]{.49\textwidth}
       \begin{tikzpicture}
 \draw[->] (-.75,0) -- (5,0) node[below , font=\tiny] {$\vert \xi\vert$};
  \draw[->] (0,-.5) -- (0,2) node[left, font=\tiny] {$\tau$}; 
  
   \draw (1,0.1) -- (1,-0.1);
  \draw (2,0.1) -- (2,-0.1);
    \node[below, yshift= -2pt, font=\tiny] at (1,0) {$1$};
    \node[below, yshift=-2pt, font=\tiny] at (2, 0) {$2$};

    \draw (-0.1, .5) -- (0.1, .5);
    \node[left, xshift=-2pt, font= \tiny] at (0,.5) {$1/2$};
  \coordinate (A) at (0,0);
  \coordinate (B) at (2,0);
  \coordinate (C) at (1,.5);
  \coordinate (D) at (0,.5);

  \filldraw[color=Sepia, fill=Sepia!30,thick] (A)--(B)--(C)--(D)--cycle;

     
      \node[xshift=4pt, font=\small] at (7/2, 1) {$\mathcal{N}_{D^-}$};
\end{tikzpicture}
\end{subfigure}
\hfill
   \begin{subfigure}[b]{.49\textwidth}
      
      \begin{tikzpicture}
 \draw[->] (-.75,0) -- (5,0) node[below , font=\tiny] {$\vert \xi\vert$};
  \draw[->] (0,-.5) -- (0,2) node[left, font=\tiny] {$\tau$}; 
  
   \draw (1.5,0.1) -- (1.5,-0.1);
  
    \node[below, yshift= -2pt, font=\tiny] at (1.5,0) {$3/2$};

    \draw (-0.1, .75) -- (0.1, .75);
    \node[left, xshift=-2pt, font= \tiny] at (0,0.75) {$3/4$};
  \coordinate (A) at (0,0);
  \coordinate (B) at (1.5,0);
  \coordinate (C) at (0, 0.75);

  \filldraw[color=Sepia, fill=Sepia!30,thick] (A)--(B)--(C)--cycle;

     
      \node[xshift=4pt, font=\small] at (7/2, 1) {$T_2(\mathcal{N}_D)$};
\end{tikzpicture}
   \end{subfigure}\hfill
   \begin{subfigure}{.49\textwidth}
       \begin{tikzpicture}
 \draw[->] (-.75,0) -- (5,0) node[below , font=\tiny] {$\vert \xi\vert$};
  \draw[->] (0,-.5) -- (0,2) node[left, font=\tiny] {$\tau$}; 
  
   \draw (1.5,0.1) -- (1.5,-0.1);
  
    \node[below, yshift= -2pt, font=\tiny] at (1.5,0) {$3/2$};

    \draw (-0.1, 0.5) -- (0.1, 0.5);
    \node[left, xshift=-2pt, font= \tiny] at (0,0.5) {$1/2$};
  \coordinate (A) at (0,0);
  \coordinate (B) at (1.5,0);
  \coordinate (C) at (0.5, 0.5);
  \coordinate (d) at (0, 0.5);

  \filldraw[color=Sepia, fill=Sepia!30,thick] (A)--(B)--(C)--(D)--cycle;

     
      \node[xshift=4pt, font=\small] at (7/2, 1) {$T_0(\mathcal{N}_{D^-})$};
\end{tikzpicture}
   \end{subfigure}
\hfill      
   \begin{subfigure}{.49\textwidth}
       \begin{tikzpicture}
 \draw[->] (-.75,0) -- (5,0) node[below , font=\tiny] {$\vert \xi\vert$};
  \draw[->] (0,-.5) -- (0,2) node[left, font=\tiny] {$\tau$}; 
  
   \draw (.5,0.1) -- (.5,-0.1);
    \node[below, yshift= -2pt, font=\tiny] at (.5,0) {$1/2$};

    \draw (-0.1, .25) -- (0.1, .25);
    \node[left, xshift=-2pt, font= \tiny] at (0,.25) {$1/4$};
  \coordinate (A) at (0,0);
  \coordinate (B) at (.5,0);
  \coordinate (C) at (0,.25);

  \filldraw[color=Sepia, fill=Sepia!30,thick] (A)--(B)--(C)--cycle;

     
      \node[xshift=4pt, font=\small] at (3.5, 1) {$T_3(\mathcal{N}_D)=T_1(\mathcal{N}_{D^-})$};
\end{tikzpicture}
   \end{subfigure}
   \hfill
   \caption{The Newton polygons associated to certain spaces considered in the proof of Theorem \ref{thm: main_2}.
   Observe that $T_0(\mathcal{N}_{D^-})$ is a proper subset of $T_2(\mathcal{N}_D)$. In particular, we can choose $g\in H^{T_0(\mu_{D^-})}_\perp(G^{n-1})$ such that $g\notin H^{T_2(\mu_D)}_\perp(G^{n-1})$.}
   \label{fig: trace_spc_2}
\end{figure}
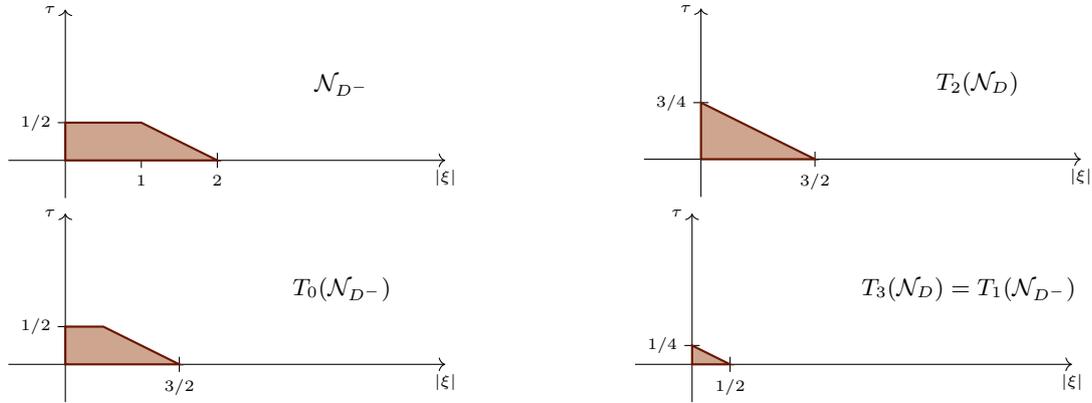
\end{proof}
We finish by providing an algebraic link between the complementing boundary condition in Definition \ref{def:compl_cond} and the classical Lopatinski\u{\i}-Shapiro condition.
Assume that we are given two boundary operators $\opB_1$ and $\opB_2$ of the form \eqref{eqn: bound_op_LS}.
Then $(\opB_1,\opB_2)$ verifies the \emph{Lopatinski\u{\i}-Shapiro conditions} for $D$, if for every fixed $(\tau, \xi')\in (\mathbb{R}\setminus\set{0})\times \mathbb{R}^{n-1}$, the polynomials $z\mapsto B_i(\tau,\xi',z):=\sum_{j=0}^{3} b_j^{(i)}(\tau,\xi')(\ic z)^j$ are linearly independent modulo $z\mapsto D^+(\tau,\xi', z)$.
\begin{proposition}\label{prop: LScond}
    For $\opB_1$ and $\opB_2$ as in \eqref{eqn: bound_op_LS}, consider the associated extended boundary matrix $\mathcal{B}$ in Definition \ref{def:compl_cond}.
    Then $\opB_1$ and $\opB_2$ verify the Lopatinski\u{\i}-Shapiro conditions for $D$ if and only if $\det \mathcal{B}(\tau, \xi')\neq 0$ for all $(\tau, \xi')\in (\mathbb{R}\setminus\set{0})\times \mathbb{R}^{n-1}$. 
\end{proposition}
\begin{proof}
    Let $(\tau, \xi')\in (\mathbb{R}\setminus\set{0})\times \mathbb{R}^{n-1}$.
    We have $\det\mathcal{B}(\tau, \xi')=0$ if and only if the rows of $\mathcal{B}(\tau, \xi')$ are linearly dependent.
    This is the case if and only if there exist $c_1,\ldots,c_4\in \mathbb{C}$ such that
    \begin{align*}
        c_1B_1(\tau,\xi',z)+c_2B_2(\tau,\xi',z)=(c_3+ c_4\ic z) D^+(\tau, \xi', z) \quad \text{for all } z\in \C,   
    \end{align*}
    that is, if and only if $B_1(\tau, \xi', \cdot)$ and $B_2(\tau,\xi', \cdot)$ are linearly dependent modulo $D^+(\tau, \xi', \cdot)$.
\end{proof}

\section*{Data Availability}
Data sharing is not applicable to this article as no datasets were generated or analysed.
\section*{Conflict of Interests}
The authors declare that there is no conflict of interests.

\end{document}